\documentclass{amsart}
	\pdfoutput=1
	\usepackage{aliascnt}
	\usepackage[colorlinks=true, linkcolor=blue, citecolor=blue]{hyperref}
	
	\newtheorem{thm}{Theorem}[section]

	\newaliascnt{lemma}{thm}  
	\newtheorem{lemma}[lemma]{Lemma}  
	\aliascntresetthe{lemma}

	\newaliascnt{prop}{thm}  
	\newtheorem{prop}[prop]{Proposition} 
	\aliascntresetthe{prop}

	\newaliascnt{defn}{thm}  
	  
	\aliascntresetthe{defn}

	\newaliascnt{cor}{thm}  
	 
	\aliascntresetthe{cor}

	\newaliascnt{rem}{thm}  
	\newtheorem{rem}[rem]{Remark} 
	\aliascntresetthe{rem}

	\newaliascnt{expl}{thm}  
	\newtheorem{expl}[expl]{Example} 
	\aliascntresetthe{expl}

	\usepackage{layout}
	
	\newcommand{\D}[1]{\operatorname{D}\left(#1\right)}
	\newcommand{\oLambda}{\operatorname{\Lambda}}
	\newcommand{\de}{\hskip.4pt\operatorname{d}\hskip-.8pt}
	
	\newcommand{\distr}[2]{$\left\langle\;#1\,,\,#2\;\right\rangle$}
	\newcommand{\ins}{\;\lrcorner\;}
	\def\cprime{$'$}
	\def\cocoa{\href{http://cocoa.dima.unige.it}{\hbox{\rm C\kern-.13em o\kern-.07em C\kern-.13em o\kern-.15em A}}}

\begin{document}

\title[Scalar Differential Invariants of Symplectic MAEs]{Scalar Differential Invariants of\\Symplectic Monge-Amp\`ere Equations}
\author{Alessandro De Paris and Alexandre M.\ Vinogradov}
\date{}
\begin{abstract}
All second order scalar differential invariants of symplectic hyperbolic and elliptic Monge-Amp\`ere equations with respect to symplectomorphisms are explicitly computed. In particular, it is shown that the number of independent second order invariants is equal to $7$, in sharp contrast with general Monge-Amp\`ere equations for which this number is equal to $2$. We also introduce a series of invariant differential forms and vector fields which allows us to construct numerous scalar differential invariants of higher order. The introduced invariants give a solution of the symplectic equivalence of Monge-Amp\`ere equations. As an example we study equations of the form $u_{xy}+f(x,y,u_x,u_y)=0$ and in particular find a simple linearization criterion.
\end{abstract}
\maketitle



\section{Introduction}

The class of Monge-Ampere equations (MAE) is, maybe, the simplest class of nonlinear PDEs which have a large spectrum of applications to geometry and mathematical physics. This class is invariant with respect to contact transformations as it was already observed by S.~Lie who set up the problem of contact classification of MAEs.  The recent progress in geometry of nonlinear PDEs revealed a high complexity of this problem which is equivalent to an explicit description of the algebra of contact scalar differential invariants of MAEs. On the other hand, in last two decades the related equivalence problem for elliptic and hyperbolic MAE was solved in the sense that differential invariants that are sufficient to distinguish two such equations were proposed. In particular, this was done in a systematic manner in the dissertation of A.~Kushner who synthesized previously proposed approaches and techniques.  His results are reported in book \cite{KLR}.  Kushner's approach is based on the use of machinery of effective forms in contact geometry, proposed to this end by V.~Lychagin at late 70s and since that time actively exploited by himself and his collaborators. At the same time an alternative approach, which is based on solution singularity theory, was proposed by the second author. One of its advantages is that it is applied to parabolic MAEs as well (see \cite{Vin2}, \cite{CV}). This approach focuses on construction of scalar differential invariants (SDI) of MAEs which are also indispensable for the classification problem. In particular, simplest SDIs, sufficient for solution of the equivalence problem, were constructed in \cite{MVY} for generic hyperbolic MAEs. These invariants are of second and third orders, i.e., depend on 2-nd and 3-rd order derivatives of coefficients of MAEs.  
This paper is a natural continuation of \cite{MVY}. We construct simplest SDIs which are sufficient for solution of the equivalence problem for non-generic elliptic and hyperbolic MAEs. Namely, we consider MAEs possessing at least one infinitesimal symmetry. In poor words, these are MAEs whose coefficients do not depend explicitly on the unknown function in a suitable local chart. Such a MAE may be naturally interpreted as a condition imposed on Lagrangian submanifolds of a symplectic 4-fold and by this reason the study of such equation reduces to some questions in symplectic geometry. For instance, a hyperbolic MAE of this kind is completely characterized by an associated 2-dimensional \hbox{non-Lagrangian} distribution $D$ on a symplectic 4-fold $M$ and its solutions are interpreted as Lagrangian submanifolds $L$ of $M$ such that the restriction of $D$ to $L$ is one-dimensional. By following the terminology of \cite{KLR} we call \emph{symplectic} MAEs of this kind (shortly, SMAEs).

The main result of this paper is an explicit construction of simplest SDIs of  SMAEs  which, in particular, are sufficient for solution of the equivalence problem. More precisely, we, first, prove that the variety of second order symplectic SDIs of non-Lagrangian 2-distributions on a symplectic 4-fold is 7-dimensional.  This result is rather surprising, since for generic MAEs the analogous variety is 2-dimensional. Then we explicitly construct nine second order SDIs for such distributions and use them to assemble seven independent second order SDIs for SMAEs. According to the principle of $n$-invariants (see \cite{ALV}, \cite{Vin3}),  one needs four independent SDIs  to solve the equivalence problem for generic SMAEs.  Therefore,  this proves that second order SDIs resolve this problem.

To our knowledge, the first solution of the equivalence problem for generic SMAEs was proposed by B.~Kruglikov in \cite{K}.  By using formalism of effective forms this author associates with a SMAE an e-structure, reducing in this way the equivalence problem for SMAEs to that for e-structures. It should be stressed that in this work SDIs of a SMAE are given in an implicit manner, namely as those of the associated e-structure. This prevents  direct manipulations with them. Also, these invariants are, mostly, of 3-rd order, i.e., not simplest ones, and it is not clear which of them are independent.  In the present paper we obtain as a byproduct six invariant vector fields. Various combinations of them give various e-structures invariantly associated with SMAEs. Those introduced by B. Kruglikov and A. Kushner are among them. Also, a machinery producing SDIs of order higher than two is briefly described. It is rather plausible that so-obtained invariants generate the whole algebra of scalar  differential invariants for SMAEs.

As an application we also discuss non-generic quasilinear SMAEs of the form $u_{xy}+f\left(x,y,u_x, u_y\right)=0$. In this case at most two of  general second order SDIs can be independent and hence additional special differential invariants are needed to solve the equivalence problem. As an illustration we construct a couple of them which allow to characterize symplectically linearizable hyperbolic MAEs. Another application of the found SDIs is a simple characterization of SMAEs which possess classical infinitesimal symmetries.

It should be stressed that scalar differential invariants of MAEs did not attract a due attention of researchers working in this field. For instance, in book \cite{KLR} just few of them are mentioned. On the other hand, SDIs are of crucial importance for the equivalence and classification problems, for practical computations of symmetries and conservation laws, etc.

The reader can find an extensive bibliography dedicated to history of topics considered in our paper in book~\cite{KLR}.

\section{Preliminaries}

\subsection{}\label{I}

The $\operatorname{C}^\infty(M)$--modules of (smooth) vector fields and $k$-forms on a smooth manifold $M$ will be denoted by $\D{M}$ and $\operatorname{\Lambda}^k(M)$, respectively. $\mathcal{L}_X$ stands for the Lie derivative along $X\in\D{M}$. Tensor products will always be understood over $\operatorname{C}^\infty(M)$. A \emph{(smooth, tangent) distribution} on $M$ we treat as a projective submodule $\mathcal{D}$ of $\D{M}$. According to Swan's theorem, such a submodule is naturally isomorphic to the module of smooth sections of a subbundle of $\operatorname{T}M$, which is more commonly taken as definition of the distribution $\mathcal{D}$. The fiber over $x\in M$ of this bundle will be denoted by $\mathcal{D}_x$. The common dimension $d$ of $\mathcal{D}_x$'s is called the dimension of $\mathcal{D}$ and we say that $\mathcal{D}$ is a $d$-distribution.

If vector fields $X_1,\ldots ,X_r$ (locally) generate a distribution $\mathcal{D}$ we write $\mathcal{D}=\left\langle X_1,\ldots ,X_r\right\rangle$. The insertion of a vector field $X$ into a $k$-form~$\alpha$ will be denoted or by $\operatorname{i}_X(\alpha)$, or by $X\ins\alpha$, i.e.,
\begin{multline*}
\operatorname{i}_X(\alpha)\left(X_1,\ldots ,X_{k-1}\right)=\left(X\ins\alpha\right)\left(X_1,\ldots ,X_{k-1}\right):=\alpha\left(X,X_1,\ldots ,X_{k-1}\right)\;,\\ X_i\in\D{M}\;.
\end{multline*}
The module of multi-vector fields of multiplicity $r$ (shortly, $r$-vectors) on $M$ will be denoted by $\operatorname{D}_r(M)$, i.e., $\operatorname{D}_r(M)=\bigwedge^r\D{M}$. If $W\in\operatorname{D}_r(M)$, $\alpha\in\operatorname{\Lambda}^k(M)$, then $W\ins\alpha\in\operatorname{\Lambda}^{k-r}(M)$ stands for the insertion of $W$ into $\alpha$. Recall that for $W=X_1\wedge\cdots\wedge X_r$ 
\[
\left(X_1\wedge\cdots\wedge X_r\right)\ins\alpha=X_r\ins\left(\ldots\ins \left(X_1\ins\alpha\right)\ldots\right)\:.
\]
Our construction of differential invariants needs \emph{vector valued differential $k$-forms}, that is, elements of the $\operatorname{C}^\infty(M)$--module $\operatorname{\Lambda}^k(M)\otimes\D{M}$. Since $\operatorname{\Lambda}^k(M)$, $\D{M}$ are projective modules, such forms may alternatively be understood as alternating functions of $k$ vector fields with values in vector fields (hence the name). For instance, $\operatorname{\Lambda}^1(M)\otimes\D{M}$ is naturally identified with $\operatorname{End}\left(\D{M}\right)$. The insertion of a vector-valued form $\omega\in\operatorname{\Lambda}^k(M)\otimes\D{M}$ into a form $\beta\in\operatorname{\Lambda}^r(M)$ will be denoted by $\omega\ins\beta\in\operatorname{\Lambda}^{r+k-1}(M)$. If $\omega=\alpha\otimes X$, then
\[
(\alpha\otimes X)\ins\beta=\alpha\wedge(X\ins\beta)\;.
\]

The explicit formula for $\omega\ins\beta$ is
\begin{multline}
\label{Formula}
\left(\omega\ins\beta\right)  \left(  X_{1},\ldots,X_{k+r-1}\right)=\\\sum_{\sigma\in S_{k,r-1}}\left(-1\right)^{\left| \sigma\right|  }
\beta\left(\omega\left(  X_{\sigma(1)},\ldots,X_{\sigma(k)}\right),  X_{\sigma(k+1)},\ldots ,X_{\sigma(k+r-1)}\right)\;,\\
X_{1},\ldots,X_{k+r-1}\in\D{M}\;,
\end{multline}
with $S_{k,r-1}$ being the set of permutations such that
\[
\sigma\left(  1\right)  <\cdots<\sigma\left(  k\right)\qquad\text{and}\qquad\sigma\left( k+1\right)<\cdots<\sigma\left(  k+r-1\right)\,,
\]
and $\left| \sigma\right|$ stands for the parity of $\sigma$.

\subsection{}

Let $\Omega$ be a symplectic form on a $2n$-dimensional manifold $M$. The isomorphism of $\operatorname{C}^{\infty}(M)$--modules
\[
\Gamma:\D{M}\overset{\sim}{\longrightarrow}\operatorname{\Lambda}^1(M)\;,\qquad X\mapsto X\ins\Omega
\]
is naturally associated with $\Omega$. Since
\[
\Omega\left(X,Y\right)=\Gamma\left(X\right)(Y)\;,
\]
$\Gamma$ uniquely determines $\Omega$. The isomorphism $\Gamma$ naturally extends  to an isomorphism of exterior algebras $D_{\ast}(M)\to\operatorname{\Lambda}^\ast(M)$, which will still be denoted by $\Gamma$:
\[
\Gamma\left(X_1\wedge\cdots\wedge X_k\right)=\Gamma\left(X_1\right)\wedge\cdots\wedge\Gamma\left(X_k\right)\;,\quad X_1,\ldots ,X_k\in\D{M}\;.
\]

Furthermore, $\Omega$ extends to $\operatorname{C}^{\infty}(M)$--bilinear forms on $D_k(M)$ and $\operatorname{\Lambda}^k(M)$:
\[
\langle A,B\rangle:=B\ins\Gamma(A)\;,\quad\langle\alpha ,\beta\rangle:=\Gamma^{-1}(\beta)\ins\alpha\;,\quad A,B\in\operatorname{D}_k(M)\;,\quad\alpha ,\beta\in\operatorname{\Lambda}^k(M)\;.
\]

These bilinear forms are graded-symmetric:
\[
\langle a,b \rangle=(-1)^k\langle b,a \rangle\;.
\]

The condition $\langle \mathcal{V}, \mathcal{V}\rangle=1$ on volume forms $\mathcal{V}\in\operatorname{\Lambda}^{2n}(M)$, together with the symplectic canonical orientation given by $\Omega^n$, select a privileged volume form, which turns out to be $ \mathcal{V}_{\Omega}:=(1/n!)\Omega^n$. The \emph{symplectic Hodge star}, denoted by $\ast$, is the operator $\operatorname{\Lambda}^k(M)\to\operatorname{\Lambda}^{2n-k}(M)$ uniquely defined by the condition
\[
\alpha\wedge\ast\beta=\langle\alpha,\beta\rangle \mathcal{V}_{\Omega}\,,\qquad\forall\alpha,\beta\in\operatorname{\Lambda}^k(M)\;.
\]

In this paper we shall be concerned with the case $n=2$ only. Canonical local coordinates for $\Omega$ will be denoted by $x,p,y,q$, i.e., locally, $\Omega=\de p\wedge\de x+\de q\wedge\de y$.

\subsection{Contact Manifolds}

Recall that a \emph{contact manifold} is a pair $\left(N,\mathcal{C}\right)$, where $N$ is an odd-dimensional manifold, say, $\dim N=2n+1$, and $\mathcal{C}$ is a `completely non-integrable' $2n$-distribution on $N$. This means that $\mathcal{C}$ does not admit nonzero characteristics, i.e., vector fields $X\in\mathcal{C}$ whose flow leaves $\mathcal{C}$ invariant. Locally $\mathcal{C}$ can be defined by an annihilating it form $\omega\in\operatorname{\Lambda}^1(N)$, i.e., $X\in\mathcal{C}\iff\omega(X)=0$ \footnote{Here and sometimes in the following we omit notation for restrictions onto open subsets, when dealing with local constructions.}. If $\omega$ is such a form, then $\left(\de\omega\right)^n\wedge\omega$ is nowhere $0$. Two vector fields $X,Y\in\mathcal{C}$ are \emph{$\mathcal{C}$-orthogonal} if $[X,Y]\in\mathcal{C}$. As it is easy to see, this is equivalent to $\de\omega(X,Y)=0$. Moreover, there exists a skew-symmetric $\operatorname{C}^{\infty}(N)$--bilinear form $\Theta$ on $\mathcal{C}$ such that $\Theta(X,Y)=0$ iff $X$ and $Y$ are $\mathcal{C}$-orthogonal. Such a form is unique up to a nowhere vanishing factor $f\in\operatorname{C}^{\infty}(N)$. For instance, $\left.\de\omega\right|_{\mathcal{C}}$ is locally such a form.

An $n$-dimensional submanifold $L\subseteq N$ is called \emph{Legendrian} if any two tangent to $L$ vectors are $\mathcal{C}$-orthogonal, i.e., $\left.\Theta\right|_L=0$.

\subsection{Contact Fields}
\label{CF}
A field $X\in\D{N}$ is called \emph{contact} if $[X,\mathcal{C}]\subseteq\mathcal{C}$, or, equivalently, $\mathcal{L}_X(\omega)=\lambda\omega$, $\lambda\in\operatorname{C}^{\infty}(N)$. Let $\nu_{\mathcal{C}}:=\D{N}/\mathcal{C}$. Recall that a contact field is uniquely characterized by its \emph{generating function} $F=X\!\!\!\mod\mathcal{C}\in\nu_{\mathcal{C}}$ and, conversely, to any $F\in\nu_{\mathcal{C}}$ an unique contact field denoted by $X_F$ corresponds (see \cite{KV}). Locally, a contact form $\omega$ establishes an isomorphism
\[
\nu_{\mathcal{C}}\overset{\sim}{\to}\operatorname{C}^{\infty}(N)\;,\quad X\!\!\!\!\mod\mathcal{C}\mapsto\omega(X)\;.
\]

Let $X$ be a nowhere vanishing contact field. Then trajectories of $X$ foliate $N$. Locally this foliation can be viewed as a one-dimensional fiber bundle $\sigma$ with a $2n$-dimensional base $M$. In this situation, $M$ is naturally supplied with a symplectic structure. Indeed, normalize a contact form $\omega$ by the condition $\omega(X)=1$. Then $\mathcal{L}_X(\omega)=0$ and $X\ins\de\omega=0$. These two conditions imply that $\de\omega=\sigma^\ast\left(-\Omega\right)$ where $\Omega\in\operatorname{\Lambda}^2(M)$. $\Omega$ is obviously closed and nondegenerate and hence is a symplectic form on $M$. The so-obtained pair $(M,\Omega)$ will be called the \emph{symplectic quotient of $(N,\mathcal{C})$ along $X$}.

The inverse procedure of \emph{contactization} of the symplectic manifold $(M,\Omega)$ may be \emph{locally} defined as follows. Set $N=M\times\mathbb{R}$, choose one of the many primitives that $\Omega$ locally admits, say $\rho$ and set $\omega=\pi_{\mathbb{R}}^\ast\de u-\pi_M^\ast\rho$. Here $u$ denotes the canonical coordinate function on $\mathbb{R}$ and $\pi_{\mathbb{R}},\pi_M$ stand for projections of $M\times\mathbb{R}$ onto $\mathbb{R}$ and $M$, respectively. By construction, a natural projection $\pi_M$ induces an isomorphism of the contact plane $\mathcal{C}_x$ and~$\operatorname{T}_{\pi_M(x)}M$.

\subsection{Jets}

In what follows, $\operatorname{J}^k(E,n)$ denotes the manifold of $k$-jets of $n$-subman\-ifolds of an $(n+m)$-manifold $E$ (see \cite[n.~0.2]{Vin}). When $E$ is fibered by $\pi:E\to B$, we also consider the submanifold $\operatorname{J}^k(\pi)\subseteq \operatorname{J}^k(E,n)$ of jets of sections. Recall that $\operatorname{J}^1\left(E^{n+1},n\right)$ has a canonical contact structure given by the \emph{Cartan distribution} (see \cite[n.~0.3, Example]{Vin}). Recall that if $\mathbf{1}_B:B\times\mathbb{R}\to B$ is the projection onto $B$, then there is a canonical projection $\operatorname{J}^1\left(\mathbf{1}_B\right)\to\operatorname{T}^\ast B$, which sends Legendrian submanifolds of $\operatorname{J}^1\left(\mathbf{1}_B\right)$ to Lagrangian submanifolds of $\operatorname{T}^\ast B$. A local chart $\left(x_1,\ldots ,x_n\right)$ on $B$ induces \emph{canonical coordinates} $\left(x_1,\ldots ,x_n,u,p_1,\ldots ,p_n\right)$ on $\operatorname{J}^1\left(\mathbf{1}_B\right)$. Recall also that there is a canonical contact form $\omega$ on $\operatorname{J}^1\left(\mathbf{1}_B\right)$ which in these coordinates reads as $\de u-\sum p_i\de x_i$. This form allows one to identify generating functions of contact fields on $\operatorname{J}^1\left(\mathbf{1}_B\right)$ with usual ones. In particular, the contact field corresponding to the constant function $1$ is $X_1=\partial/\partial u$. This make evident that the canonical projection $\operatorname{J}^1\left(\mathbf{1}_B\right)\to\operatorname{T}^\ast B$ is the symplectic quotient of $\operatorname{J}^1\left(\mathbf{1}_B\right)$ along $X_1$.

\subsection{Monge-Amp\`ere Equations}
\label{MAE}
Equations of the form
\begin{equation}
\label{MAcoordinata}
S\left(u_{xx}u_{yy}-u_{xy}^2\right)+Au_{xx}+Bu_{xy}+Cu_{yy}+D=0\;,
\end{equation}
with $u(x,y)$ being the unknown function and $S$, $A$, $B$, $C$, $D$ being functions of $x$, $y$, $u$, $u_x$, $u_y$, are usually called \emph{Monge-Amp\`ere equations} (\emph{MAE}, for short). We refer to them as \emph{classical} since this term is also used for their analogues in higher dimensions~\footnote{Sometimes the term `classical' refers, more restrictively, to the equation $\operatorname{det }\operatorname{Hess }u=1$.}. Geometrically, relation \eqref{MAcoordinata} is interpreted as a hypersurface in $\operatorname{J}^2\left(E^3,2\right)$.

Recall that \eqref{MAcoordinata} is hyperbolic (resp., parabolic, elliptic) if $\Delta>0$ (resp., $\Delta=0$, $\Delta<0$), where $\Delta:=B^2-4AC+4SD$. Equation \eqref{MAcoordinata} may be viewed as the analytical description of a geometric problem which, for hyperbolic equations is as follows.

Let $N$ be a contact manifold and $\mathbf{D}$ a two-dimensional non-Lagrangian distribution, i.e., $\mathbf{D}_x$ is not a Lagrangian subspace of $\mathcal{C}_x$ \emph{for all} $x\in N$. The geometrical problem is to find Legendrian submanifolds $L\subset N$ such that $\operatorname{T}_xL\cap\mathbf{D}_x$ is one-dimensional for all $x$. It is easy to see that this condition is equivalent to one-dimensionality of $\operatorname{T}_xL\cap\mathbf{D}_x'$, where $\mathbf{D}_x'$ stands for the $\Theta_x$-orthogonal complement of $\mathbf{D}_x$. This interpretation comes form the theory of singularities of multivalued solutions of PDE's and distinguishes MAEs by the nature of singularities their solutions admit. Recall that one of the distributions $\mathbf{D}$, $\mathbf{D}'$ for hyperbolic equation \eqref{MAcoordinata} is
\begin{equation}
\label{FormuleD}
\left\langle\; X-\frac{\sqrt{\Delta}}{2S}\frac{\partial}{\partial q}\quad,\quad Y+\frac{\sqrt{\Delta}}{2S}\frac{\partial}{\partial p}\;\right\rangle\;,
\end{equation}
while the other is
\[
\left\langle\; X+\frac{\sqrt{\Delta}}{2S}\frac{\partial}{\partial q}\quad,\quad Y-\frac{\sqrt{\Delta}}{2S}\frac{\partial}{\partial p}\;\right\rangle\;,
\]
where
\[
X:=\frac{\partial}{\partial x}\;+\;p\,\frac{\partial}{\partial u}\;-\;\frac{C}{S}\,\frac{\partial}{\partial p}\;+\;\frac{B}{2S}\,\frac{\partial}{\partial q}\;,\qquad Y:=\frac{\partial}{\partial y}\;+\;q\,\frac{\partial}{\partial u}\;+\;\frac{B}{2S}\,\frac{\partial}{\partial p}\;-\;\frac{A}{S}\,\frac{\partial}{\partial q}\;,
\]
assuming that $S\ne 0$. If \eqref{MAcoordinata} is quasilinear, i.e., $S=0$, these distributions are
\begin{multline*}
\left\langle\; X-\sqrt{\Delta}\frac{\partial}{\partial y}-q\sqrt{\Delta}\frac{\partial}{\partial u}\quad,\quad Y+\sqrt{\Delta}\frac{\partial}{\partial p}\;\right\rangle\;,\\\left\langle\; X+\sqrt{\Delta}\frac{\partial}{\partial y}+q\sqrt{\Delta}\frac{\partial}{\partial u}\quad,\quad Y-\sqrt{\Delta}\frac{\partial}{\partial p}\;\right\rangle\,,
\end{multline*}
with
\[
X:=2A\frac{\partial}{\partial x}\;+\;B\frac{\partial}{\partial y}\;+\;\left(2pA+qB\right)\frac{\partial}{\partial u}\;-\;2D\frac{\partial}{\partial p}\;,\qquad Y:=B\frac{\partial}{\partial p}\;-\;2A\frac{\partial}{\partial q}\;,
\]
assuming that $A\ne 0$. For this and further details, see \cite[n.~3.2]{MVY}.

A more satisfactory way to describe this situation is in terms of an operator $\mathbf{A}:\mathcal{C}\to\mathcal{C}$, such that
\begin{enumerate}
\item $\mathbf{A}^2=\operatorname{id}$, but $\mathbf{A}\ne\pm\operatorname{id}$;
\item\label{Sa} $\mathbf{A}$ is selfadjoint with respect to the bilinear form $\Theta$ on $\mathcal{C}$.
\end{enumerate}
Condition \autoref{Sa} means that $\Theta\left(\mathbf{A}(X),Y\right)=\Theta\left(X,\mathbf{A}(Y)\right)$, $\forall X,Y\in\mathcal{C}$. Similarly, an elliptic (resp., parabolic) MAE can be described in terms of a $\Theta$-selfadjoint operator such that $\mathbf{A}^2=-\operatorname{id}$ (resp., $\mathbf{A}^2=0$, $\mathbf{A}\ne 0$).  With such an operator $\mathbf{A}$ is associated the geometrical problem of finding Legendrian submanifolds $L\subset N$ such that $\operatorname{T}_xL$ is an invariant subspace of $\mathbf{A}_x:\mathcal{C}_x\to\mathcal{C}_x$, $\forall x\in L$. Such Legendrian submanifolds will be called \emph{$\mathbf{A}$-invariant}. Analytically, $\mathbf{A}$-invariant submanifolds are described as solutions of a MAE and vice versa. In the sequel we understand a MAE as a problem of finding $\mathbf{A}$-invariant Legendrian submanifolds for a given operator $\mathbf{A}$ of the above type. For elliptic and hyperbolic equations this operator is unique up to the sign. The hyperbolic MAE is in this sense associated with the operator $\mathbf{A}$ for which $\mathbf{D}$, $\mathbf{D}'$ are the root spaces corresponding to eigenvalues $1$, $-1$ \footnote{Operator $\mathbf{A}$ is considered in \cite{KLR} and some preceding publications in the context of effective differential forms approach, but not as a definition of MAEs.}. In this article we search for basic scalar differential invariants of such hyperbolic and elliptic MAEs which admit an infinitesimal symmetry. Such a symmetry $X$ is a (nontrivial) contact field whose flow consists of contact diffeomorphisms preserving $\mathbf{D}$ (or $\mathbf{D}'$), or, equivalenty, the operator $\mathbf{A}$. This is equivalent to $[X,\mathbf{D}]\subseteq\mathbf{D}$. The symplectic quotient along such a symmetry (locally) projects this situation onto the symplectic manifold $(M,\Omega)$ (see \hyperref[CF]{n.~\ref{CF}}). In particular, the distributions $\mathbf{D}$ and $\mathbf{D}'$ project onto distributions $\mathcal{D}$ and $\mathcal{D}'$, respectively, and Legendrian submanifolds in $N$ to Lagrangian submanifolds in $M$. In other words, the original MA problem projects to the following one: given a two-dimensional non-Lagrangian distribution $\mathcal{D}$ on a symplectic manifold $(M,\Omega)$, find Lagrangian submanifolds $L\subset M$ such that $\operatorname{T}_xL\cap\mathcal{D}_x$ is one-dimensional for all $x$.

By a symplectic hyperbolic MAE we understand analytical description of such a problem. If coefficients $S,\ldots , D$ in \eqref{MAcoordinata} do not depend on $u$, then $\partial/\partial u$ is a symmetry of this equation. Some authors refer to this situation as a symplectic Monge-Amp\`ere equation. It is worth stressing that from \eqref{MAcoordinata} it is not clear which, contact or symplectic, MAE it expresses. Accordingly, we have to distinguish contact differential invariants from symplectic ones. A natural relation between them will be explained below. Also it should be stressed that a symplectic MAE can be obtained from a contact one. Namely, consider the contactization of $M$ (see \hyperref[CF]{n.~\ref{CF}}) and observe that there is a unique bidimensional distribution $\mathbf{D}\subset\mathcal{C}$, with $\mathcal{C}$ being the contact distribution, that projects onto $\mathcal{D}$. Indeed, $\mathbf{D}_y=\left(\left.\de\pi_M\right|_{\mathcal{C}_y}\right)^{-1}(\mathcal{D}_x)$, $x=\pi(y)$.

\subsection{Bundles of Equations}
\label{BoE}
MAEs on a symplectic manifold $(M,\Omega)$ can be locally identified with sections of a trivial projective bundle $\pi: M\times\mathbb{P}^4\mapsto M$. Indeed, an independent from $u$ local representation \eqref{MAcoordinata} gives rise to a local section
\[
M\to M\times\mathbb{P}^4\;,\quad \mathbf{p}\mapsto \left(\mathbf{p},\left[S\left(\mathbf{p}\right),A\left(\mathbf{p}\right),B\left(\mathbf{p}\right),C\left(\mathbf{p}\right),D\left(\mathbf{p}\right)\right]\right)\;.
\]

In view of the interpretation of symplectic hyperbolic MAEs as pairs of distributions, it is also convenient to represent (single) $2$-distributions on $M$ by a bundle $\gamma:G\to M$ whose fiber at $\mathbf{p}\in M$ is the Grassmannian $\operatorname{G}_2\left(\operatorname{T}_{\mathbf{p}}M\right)$. This way one gets a two-fold covering of the hyperbolic open subset of $M\times\mathbb{P}^4$ by the non-Lagrangian open subset of $G$.

To introduce a convenient local chart in $\pi: M\times\mathbb{P}^4\mapsto M$ and its jet powers, we consider the standard open affine subset given by points with nonzero first projective coordinate and a canonical chart $(x,p,y,q)$ on $M$. By denoting these affine coordinates by $v^1, \ldots, v^4$, one gets a chart $\left(x,p,y,q,v^1,\ldots ,v^4\right)$ in $M\times\mathbb{P}^4$. In other words, if $\mathcal{E}$ is given by \eqref{MAcoordinata} then the corresponding local section of $\pi$ is given by
\[
\mathbf{p}\mapsto\left(\;\mathbf{p}\,,\,\frac {A(\mathbf{p})}{S(\mathbf{p})}\,,\,\frac {B(\mathbf{p})}{S(\mathbf{p})}\,,\,\frac {C(\mathbf{p})}{S(\mathbf{p})}\,,\,\frac {D(\mathbf{p})}{S(\mathbf{p})}\;\right)\;.
\]

Similarly, we define a local chart $\left(x,p,y,q,u^1,\ldots ,u^4\right)$ in $G$ in such a way the aforementioned two-fold covering is described by
\[
v^1=-u^4\;,\quad v^2=u^2+u^3\;,\quad v^3=-u^1\;,\quad v^4=u^1u^4-u^2u^3
\]
and one of its two (continuos right-) inverse maps by
\[
u^1=-v^3\;,\quad u^2=\frac{v^2-\sqrt{\Delta}}2\;,\quad u^3=\frac{v^2+\sqrt{\Delta}}2\;,\quad u^4=-v^1\;,
\]
with $\Delta:=\left(v^2\right)^2-4v^1v^3+4v^4$ (cf.~\eqref{FormuleD};  see also \cite[n.~3.3.1]{MVY}).

Scalar $k$-th order differential invariants of symplectic MAEs and $2$-distributions can be understood as functions (locally) defined on $\operatorname{J}^k(\pi)$ and $\operatorname{J}^k(\gamma)$, respectively, that are invariant under a natural action of symplectomorphisms. In the following exposition, we do not need an explicit description of this action and, so, it is omitted. If $\mathcal{I}$ is such a function and $s$ is a representing section of $\mathcal{E}$ (resp., $\mathcal{D}$), then we set
\[
\mathcal{I}_{\mathcal{E}}:=\mathcal{I}\circ\operatorname{J}^k(s)\;,\qquad\left(\text{resp., }\mathcal{I}_{\mathcal{D}}:=\mathcal{I}\circ\operatorname{J}^k(s)\right)\;.
\]
$\mathcal{I}_{\mathcal{E}}$ (resp., $\mathcal{I}_{\mathcal{D}}$) is called the \emph{value} of $\mathcal{I}$ on $\mathcal{E}$ (resp., on $\mathcal{D}$). Obviously, a differential invariant $\mathcal{I}$ can be defined by explicitly describing its values $\mathcal{I}_{\mathcal{E}}$ (resp, $\mathcal{I}_{\mathcal{D}}$). Below we follow this approach.

\section{\texorpdfstring{Differential Invariants of non-Lagrangian $2$-Distributions\\in Symplectic $4$-folds}{Differential Invariants of non-Lagrangian Two-Distributions in Symplectic Four-folds}}\label{DIDS}

Throughout this section $\left(M,\Omega\right)$ stands for a symplectic $4$-fold and $\mathcal{D}$ for a non-Lagrangian $2$-distribution on $M$. Denote by $\mathcal{D}'$ the $\Omega$-orthogonal complement of~$\mathcal{D}$. Obviously,
\begin{equation}
\label{split}
\mathcal{D}\oplus\mathcal{D}'=\D{M}\;.
\end{equation}
Denote by $P:\D{M}\to\D{M}$ and $P'=\operatorname{id}-P$ the corresponding projections onto $\mathcal{D}$ and $\mathcal{D}'$, respectively.

In this section we shall deduce some basic scalar differential invariants of the geometrical structure $\left(\Omega,\mathcal{D}\right)$ over~$M$. Observe that there is a natural bijection of differential invariants of $\left(\Omega,\mathcal{D}\right)$ and $\left(\Omega,\mathcal{D}'\right)$. Namely, with a given differential invariant $\mathcal{I}$ is naturally associated a differential invariant $\mathcal{I}'$, such that $\mathcal{I}'_{\mathcal{D}}=\mathcal{I}_{\mathcal{D}'}$. This way one gets an involution $\mathfrak{I}$ acting on differential invariants of $\left(\Omega,\mathcal{D}\right)$. We start describing some non-scalar differential invariants by means of which we shall construct some \emph{scalar} ones.

The first such invariant is the vector field
\[
Z:=P'\left(\left[X,Y\right]\right),\quad\text{with } X,Y\in\mathcal{D},\;\Omega\left(X,Y\right)=1\;.
\]

\begin{lemma}\label{ZWD}
\ 
\begin{enumerate}
\item $Z$ is well defined;
\item\label{Ovvia} $Z\in\mathcal{D}'$.
\end{enumerate}
\end{lemma}
\begin{proof}
Observe that $P'\left(\left[fX,Y\right]\right)=fP'\left(\left[X,Y\right]\right)$, since $P'(X)=0$, and, similarly, for $Y$. If $\overline{X},\overline{Y}\in\mathcal{D}$ are such that $\Omega\left(\overline{X},\overline{Y}\right)=1$, then $\overline{X}=\alpha X+\beta Y$, $\overline{Y}=\gamma X+\delta Y$, with $\left|\begin{array}{cc}\alpha&\beta\\\gamma&\delta\end{array}\right|=1$. So, the first assertion directly follows from these two facts. The second assertion is obvious.
\end{proof}

Also, put $Z':=\mathfrak{I}(Z)$, $Z'\in\mathcal{D}$. By using splitting \eqref{split} define $\omega\in\oLambda^2\left(M\right)$ by conditions
\[
\left.\omega\right|_{\mathcal{D}}=\left.\Omega\right|_{\mathcal{D}},\quad\operatorname{Ker }\omega=\mathcal{D}'\;.
\]
If $\omega'=\mathfrak{I}(\omega)$, then, obviously, $\omega'=\Omega-\omega$, i.e.,
\begin{equation}
\label{Splitfs}
\Omega=\omega+\omega'\;.
\end{equation}
\begin{lemma}\label{Ins}
\ 
\[
\omega=\frac12 P\ins\Omega\;,\quad\text{(resp., }\omega'=\frac12 P'\ins\Omega\text{ )}\;.
\]
\end{lemma}
\begin{proof}
According to \eqref{Formula} we have
\[
\left(P\ins\Omega\right)\left(X,Y\right)=\Omega\left(P(X),Y\right)+\Omega\left(X,P(Y)\right)\;.
\]
This easily implies that $\frac12 P\ins\Omega$ satisfies the defining conditions of $\omega$.
\end{proof}

Define the \emph{curvature} $\mathcal{R}$ of $\mathcal{D}$ by
\[
\mathcal{R}\left(X,Y\right):=P'\left(\left[P(X),P(Y)\right]\right)\;,\quad X,Y\in\D{M}\;.
\]
Obviously, $\mathcal{R}':=\mathfrak{I}(\mathcal{R})$ is the curvature of $\mathcal{D}'$.

\begin{lemma}\label{ReZ}
\ 
\begin{enumerate}
\item $\mathcal{R}$ (resp., $\mathcal{R}'$) is skew-symmetric and $\operatorname{C}^\infty(M)$-bilinear; 
\item\label{Decomp} $\mathcal{R}=\omega\otimes Z$ (resp., $\mathcal{R}'=\omega'\otimes Z'$);
\item $\left(\mathcal{R}\ins\Omega\right)\left(X,Y,V\right)=\Omega\left(\mathcal{R}\left(X,Y\right),V\right)+\Omega\left(\mathcal{R}\left(V,X\right),Y\right)+\Omega\left(\mathcal{R}\left(Y,V\right),X\right)$.
\end{enumerate}
\end{lemma}
\begin{proof}
\begin{enumerate}
\item Skew-symmetry is obvious. $\operatorname{C}^\infty(M)$-bilinearity is as in the proof of~\autoref{ZWD}.
\item When $X,Y\in\mathcal{D}$ and $\Omega(X,Y)=1$ we have
\begin{multline*}
\mathcal{R}\left(X,Y\right)=P'\left(\left[P(X),P(Y)\right]\right)=P'\left(\left[X,Y\right]\right)=Z=\Omega(X,Y)Z=\omega(X,Y)Z\\=\left(\omega\otimes Z\right)(X,Y)\;;
\end{multline*}
moreover, the equality is trivial when $X,Y\in\mathcal{D}$ and $\Omega(X,Y)=0$.

When $X\in\mathcal{D}'$ we have
\[
\mathcal{R}\left(X,Y\right)=P'\left(\left[P(X),P(Y)\right]\right)=P'\left(\left[0,Y\right]\right)=0
\]
and
\[
\left(\omega\otimes Z\right)(X,Y)=\omega(X,Y)Z=0\;.
\]

Now, the general result easily follows by from these two facts, $\operatorname{C}^\infty(M)$--linearity and the splitting $\mathcal{D}\oplus\mathcal{D}'=\D{M}$.
\item Straightforwardly from \eqref{Formula} (\hyperref[I]{n.~\ref{I}}).
\end{enumerate}
\end{proof}

\begin{lemma}\label{Cancel}
\ 
\[
\omega\wedge\left(Z'\ins\omega\right)=\omega'\wedge\left(Z\ins\omega'\right)=0\;.
\]
\end{lemma}
\begin{proof}
The $2$-forms $\omega$ and $\omega'$ are degenerate. Observe that a degenerate $n$-form $\alpha$ on a $2n$-dimensional manifold squares to zero. Hence for all vector fields $X$, we have $\left(X\ins\alpha\right)\wedge\alpha=X\ins\frac 12\alpha^2=0$.
\end{proof}
\begin{lemma}\label{Pre5}
\ 
\[
\left(\mathcal{R}-\mathcal{R}'\right)\ins\Omega=\left(Z-Z'\right)\ins\frac 12\Omega^2\;.
\]
\end{lemma}
\begin{proof}
First, we have
\begin{multline*}
\left(Z-Z'\right)\ins\frac 12\Omega^2-\left(\mathcal{R}-\mathcal{R}'\right)\ins\Omega\\\overset{\autoref{ReZ},\ \autoref{Decomp}}{=}\;\Omega\wedge\left(\left(Z-Z'\right)\ins\Omega\right)-\left(\omega\otimes Z-\omega'\otimes Z'\right)\ins\Omega\\
=\Omega\wedge\left(Z\ins\Omega\right)-\Omega\wedge\left(Z'\ins\Omega\right)-\omega\wedge\left(Z\ins\Omega\right)+\omega'\wedge\left(Z'\ins\Omega\right)
\overset{\eqref{Splitfs}}{=}\;\omega'\wedge\left(Z\ins\Omega\right)-\omega\wedge\left(Z'\ins\Omega\right)\\
\overset{\eqref{Splitfs},\,\autoref{Cancel}}{=}\;\omega'\wedge\left(Z\ins \omega\right)-\omega\wedge\left(Z'\ins \omega'\right)\;.
\end{multline*}
But by \autoref{ZWD}, \autoref{Ovvia}, $Z\ins \omega=0$, and, similarly, $Z'\ins \omega'=0$.
\end{proof}

\begin{prop}\label{N5}
\ 
\[
\left(Z-Z'\right)\ins\frac 12\Omega^2=\de\omega\;.
\]
\end{prop}

\begin{proof}
It suffices to prove this formula for open subsets $U$ such that $\left.\mathcal{D}\right|_{U}$ admits spanning vector fields $X,Y$. These can be normalized to
\begin{equation}
\label{Normalizzati}
\Omega\left(X,Y\right)=1\;.
\end{equation}
Also make a similar choice of vector fields $X',Y'$ for $\mathcal{D}'$.

Since $X,Y,Z'\in\mathcal{D}$ and $\mathcal{D}$ is $2$-dimensional, $X,Y,Z'\in\mathcal{D}$ are $\operatorname{C}^\infty(M)$--dependent. So,
\[
Y\ins\left(X\ins\left(\left(Z-Z'\right)\ins\frac 12\Omega^2\right)\right)=Y\ins\left(X\ins\left(Z\ins\frac 12\Omega^2\right)\right)\;.
\]
Taking into account \eqref{Splitfs} and the fact that $X,Y\in\operatorname{Ker}\omega'$, $Z\in\operatorname{Ker}\omega$, we have
\begin{multline*}
Y\ins\left(X\ins\left(Z\ins\frac 12\Omega^2\right)\right)=Y\ins\left(X\ins\left(\Omega\wedge\left(Z\ins\Omega\right)\right)\right)=\left(Y\ins\left(X\ins\Omega\right)\right)\wedge\left(Z\ins\omega'\right)\\
\overset{\eqref{Normalizzati}}{=}Z\ins\omega'\;.
\end{multline*}
Taking into account that $X$, $Y$ and $\left[X,Y\right]-Z$ belong to the kernel of $\omega'$, we have
\begin{multline*}
Z\ins\omega'=\left[X,Y\right]\ins\omega'=X\ins\mathcal{L}_Y\omega'-\mathcal{L}_Y\left(X\ins\omega'\right)=X\ins\mathcal{L}_Y\omega'\\
=X\ins\left(Y\ins\de\omega'\right)+X\ins\de\left(Y\ins\omega'\right)=X\ins\left(Y\ins \de\omega'\right)
\overset{\eqref{Splitfs}}{=}Y\ins\left(X\ins\de\omega\right)\;.
\end{multline*}
Thus we have shown that
\begin{equation}
\label{Prima}
Y\ins\left(X\ins\left(\left(Z-Z'\right)\ins\frac 12\Omega^2\right)\right)=Y\ins\left(X\ins \de\omega\right)\;.
\end{equation}

Similar arguments applied to $\mathcal{D}'$, $X'$ and $Y'$ give
\[
Y'\ins\left(X'\ins\left(\left(Z'-Z\right)\ins\frac 12\Omega^2\right)\right)=Y'\ins\left(X'\ins \de\omega'\right)\;,
\]
or, equivalently,
\begin{equation}
\label{Seconda}
Y'\ins\left(X'\ins\left(\left(Z-Z'\right)\ins\frac 12\Omega^2\right)\right)=Y'\ins\left(X'\ins \de\omega\right)\;.
\end{equation}

Identities~\eqref{Prima} and~\eqref{Seconda} imply the required identity, because the insertion into a $3$-form (namely, $\left(Z'-Z\right)\ins\frac 12\Omega^2$ in the considered situation) of any three (distinct) vector fields chosen from the basis $X,Y,X',Y'$ involve either insertions of $X$ and $Y$, or insertions of~$X'$ and~$Y'$.
\end{proof}

The following invariant differential forms
\[
\rho:=\Gamma(Z)=Z\ins\Omega\,,\qquad\rho':=\Gamma\left(Z'\right)=Z'\ins\Omega\,,\qquad\sigma:=\rho-\rho'\,
\]
will be used in our construction of scalar differential invariants.

We have constructed the following invariants $1$- and $2$-forms: $\rho$, $\rho'$, $\omega$, $\omega'$. Now, by making use of them, it is not difficult to construct a series of scalar differential invariants. Namely, if $\tau$ and $\tau'$ are invariant $1$-forms and $\Theta$, $\Theta'$ are invariant $2$-forms, then
\[
\ast\left(\Theta\wedge\Theta'\right)\,,\;\ast\left(\tau\wedge\tau'\wedge\Theta\right)\,,\;\ast\left(\de\tau\wedge\Theta\right)\,,\;\ast\left(\tau\wedge\de\Theta\right)\,,\;\ast\left(\de\tau\wedge\de\tau'\right)\,,\;\text{etc. }
\]
are, obviously, scalar differential invariants. However, in the considered context, the so obtained invariants are not independent. Below we shall choose, in a sense, more simple ones. The simplest of them is
\[
\mathcal{I}^1_{\mathcal D}:=\ast\left(\omega\wedge\de\sigma\right)\;.
\]
This invariant has alternative useful descriptions.
\begin{lemma}\label{DescrI1}
\ 
\begin{enumerate}
\item\label{Start} $\mathcal{I}^1_{\mathcal D}=\ast\de\left(\omega\wedge\sigma\right)$;
\item\label{Cont} $\mathcal{I}^1_{\mathcal D}=\ast\de\left(\omega\wedge\rho\right)$;
\item\label{Ponte} $\mathcal{I}^1_{\mathcal D}=\ast\left(\Omega\wedge\de\rho\right)$.
\item \label{Dausare}$\mathcal{I}^1_{\mathcal D}=\ast\de\left(\mathcal{R}\ins\Omega\right)$;
\end{enumerate}
\end{lemma}
\begin{proof}
\begin{enumerate}
\item
We have
\[
\de\left(\omega\wedge\sigma\right)=\de\omega\wedge\sigma+\omega\wedge\de\sigma\overset{\autoref{N5}}{=}\Omega\wedge\sigma\wedge\sigma+\omega\wedge\de\sigma=\omega\wedge\de\sigma\;.
\]
Hence $\mathcal{I}^1_{\mathcal D}=\ast\left(\omega\wedge\de\sigma\right)=\ast\de\left(\omega\wedge\sigma\right)$.
\item
Note that $\omega\wedge\rho'=0$, since the kernels of $\omega$ and $\rho'$ both contain the $2$-distribution $\mathcal{D}'$. Then
\[
\de\left(\omega\wedge\rho\right)=\de\left(\omega\wedge\sigma\right)\;,
\]
and the result follows from \autoref{Start}.
\item
Similarly, $\omega'\wedge\rho=0$. Hence
\[
\de\left(\omega\wedge\rho\right)\overset{\eqref{Splitfs}}{=}\de\left(\Omega\wedge\rho\right)=\Omega\wedge\de\rho\;,
\]
and the result follows from \autoref{Cont}.
\item
By \autoref{ReZ},\ \autoref{Decomp}, we have
\[
\mathcal{R}\ins\Omega=\omega\wedge\rho\;,
\]
and the result follows from \autoref{Cont}.
\end{enumerate}
\end{proof}

By rewriting the identity of \autoref{N5} as $\Omega\wedge\sigma=\de\omega$, we have $\Omega\wedge\de\sigma=\de(\Omega\wedge\sigma)=0$. Hence $\Omega\wedge\de\rho=\Omega\wedge\de\rho'$. Thus the description \autoref{Ponte} above, and consequently all the others, still hold when replacing $\omega$, $\rho$, $\mathcal{R}$ by their counterparts $\omega'$, $\rho'$, $\mathcal{R}'$ through $\mathfrak{I}$.

Other scalar differential invariants we shall deal with are
\begin{gather*}
\mathcal{I}^2_{\mathcal{D}}:=\ast\left(\sigma\wedge\rho\wedge\de\sigma\right)\;,\quad\mathcal{I}^3_{\mathcal{D}}:=\ast\left(\sigma\wedge\rho\wedge\de\rho\right)\;,\\
\mathcal{I}^4_{\mathcal{D}}:=\ast\left(\left(\de\sigma\right)^2\right)\;,\quad\mathcal{I}^5_{\mathcal{D}}:=\ast\left(\de\sigma\wedge\de\rho\right)\;,\quad\mathcal{I}^6_{\mathcal{D}}:=\ast\left(\left(\de\rho\right)^2\right)\;,\\
\mathcal{I}^7_{\mathcal{D}}:=\ast\left(\sigma\,\wedge\,\de\rho\,\wedge\,\ast\left(\sigma\wedge\de\sigma\right)\right)\;,\quad\mathcal{I}^8_{\mathcal{D}}:=\ast\left(\sigma\,\wedge\,\de\rho\,\wedge\,\ast\left(\rho\wedge\de\sigma\right)\right)\;,\\
\quad\mathcal{I}^9_{\mathcal{D}}:=\ast\left(\sigma\,\wedge\,\de\sigma\,\wedge\,\ast\left(\rho\wedge\de\rho\right)\right)\;.
\end{gather*}

It is worth noticing that obvious differential invariants $Z'\ins\rho$, $Z\ins\rho'$, and similar, are trivial.

\begin{lemma}
\ 
\begin{enumerate}
\item\label{F1} $\mathcal{I}^2_{\mathcal{D}}=\ast\left(\sigma\wedge\rho'\wedge\de\sigma\right)=\ast\left(\rho\wedge\rho'\wedge\de\sigma\right)$;
\item\label{F2}$\mathcal{I}^2_{\mathcal{D}}=-\de\sigma\left(Z,Z'\right)$;
\item\label{F3} $\mathcal{I}^3_{\mathcal{D}}=\ast\left(\sigma\wedge\rho'\wedge\de\rho\right)=\ast\left(\rho\wedge\rho'\wedge\de\rho\right)$;
\item\label{F4}$\mathcal{I}^3_{\mathcal{D}}=-\de\rho\left(Z,Z'\right)$.
\end{enumerate}
\end{lemma}
\begin{proof}
\hyperref[F1]{N.~\ref{F1}} and \autoref{F3} immediately come from $\sigma=\rho-\rho'$.

To prove \autoref{F2}, we observe that
\[
\left(Z\ins\de\sigma\right)\wedge\left(\frac12\Omega^2\right)=0
\]
as a $5$-form on a $4$-fold. Therefore, by inserting $Z'$ we obtain
\begin{equation}
\label{1descr}
\de\sigma\left(Z,Z'\right)\cdot\left(\frac12\Omega^2\right)+\left(Z'\ins\left(\frac12\Omega^2\right)\right)\wedge\left(Z\ins\de\sigma\right) =0\;.
\end{equation}
Similarly, $\left(\Omega\wedge\rho'\right)\wedge\de\sigma=0$ implies
\[
\left(Z\ins\left(\Omega\wedge\rho'\right)\right)\wedge\de\sigma=\Omega\wedge\rho'\wedge\left(Z\ins\de\sigma\right)\,,
\]
and therefore
\begin{multline}
\label{2descr}
\left(Z'\ins\left(\frac12\Omega^2\right)\right)\wedge\left(Z\ins\de\sigma\right)=\Omega\wedge\rho'\wedge\left(Z\ins\de\sigma\right)=\left(Z\ins\left(\Omega\wedge\rho'\right)\right)\wedge\de\sigma\\
=\left(Z\ins\Omega\right)\wedge\rho'\wedge\de\sigma=\rho\wedge\rho'\wedge\de\sigma\;.
\end{multline}
Now, the result immediately follows from \eqref{1descr}, \eqref{2descr} and \autoref{F1}.

To prove \autoref{F4} it suffices to replace $\de\sigma$ by $\de\rho$ in the above arguments.
\end{proof}

By using the involution $\mathfrak{I}$, we obtain a `dual' system of scalar invariants
\[
{\mathcal{I}^k}'=\mathfrak{I}\left(\mathcal{I}^k\right)\;,\qquad k=1,\ldots ,9\;.
\]
However, these are not new invariants. In particular, we have
\begin{prop}\label{Involution}
The following relations hold:
\begin{gather*}
{\mathcal{I}^1}'={\mathcal{I}^1}\;,\quad {\mathcal{I}^2}'=\mathcal{I}^2\;,\quad {\mathcal{I}^3}'=\mathcal{I}^2-\mathcal{I}^3\;,\\
{\mathcal{I}^4}'={\mathcal{I}^4}\;,\quad {\mathcal{I}^5}'=\mathcal{I}^4-\mathcal{I}^5\;,\quad {\mathcal{I}^6}'=\mathcal{I}^4-2\mathcal{I}^5+\mathcal{I}^6\;,\\
{\mathcal{I}^7}'=-{\mathcal{I}^7}\;.
\end{gather*}
\end{prop}
\begin{proof}
These formulae are more or less direct consequences of previously established relations connecting the involved invariant $1$- and $2$-forms. For instance, using the description \autoref{Dausare} in \autoref{DescrI1}, the first one immediately comes from \autoref{Pre5} and \autoref{N5}. All remaining cases easily follow from relations $\sigma=\rho-\rho'$ and $\sigma'=-\sigma$. For instance:
\begin{multline*}
{\mathcal{I}^6_{\mathcal{D}}}'=\ast\left(\left(\de\rho'\right)^2\right)=\ast\left(\left(-\de\sigma+\de\rho\right)^2\right)=\ast\left(\left(\de\sigma\right)^2-2\de\sigma\wedge\de\rho+\left(\de\rho\right)^2\right)\\
=\mathcal{I}^4_{\mathcal{D}}-2\mathcal{I}^5_{\mathcal{D}}+\mathcal{I}^6_{\mathcal{D}}\;.
\end{multline*}
\end{proof}

\section{Equivalence problem}

According to the general principle of $n$-invariants, we need four independent scalar invariants (see \cite[Chap.~7, Sect.~4.3]{ALV}). We say that some functions $\mathcal{I}^1,\ldots ,\mathcal{I}^k$ are \emph{(functionally) independent} when $\de\mathcal{I}^1,\ldots ,\de\mathcal{I}^k$ are linearly independent at every point in an open and dense subset.

\begin{prop}\label{Fourindip}
The invariants $\mathcal{I}^1,\mathcal{I}^2,\mathcal{I}^3,\mathcal{I}^5$ are independent.
\end{prop}

\begin{proof}
Let $\mathcal{D}$ be (locally) spanned by vector fields
\[
(xy+1)\partial_p+\partial_y+pq\partial_q\;,\qquad \partial_x+\partial_p+xy\partial_q
\]
(in a canonical chart). A direct calculation gives
\begin{multline*}
Z=(-xpy+x-q)\partial_x+(-xpy+xy^2+x-q)\partial_p\\+y\partial_y+(-x^2py^2+x^2y-xpy-xyq+pyq+x-q)\partial_q\;,
\end{multline*}
\begin{multline*}
Z'=(-xpy+x-p-q)\partial_x+(-xpy+xy^2+x-p+y-q)\partial_p\\+y\partial_y+(-x^2py^2+x^2y-xpy-xyq+pyq)\partial_q\;,
\end{multline*}
\begin{multline*}
\omega=(-xy-1)\de x\wedge\de p+(x^2y^2+xy-pq)\de x\wedge\de y+\de x\wedge\de q\\+pq\de p\wedge\de y-\de p\wedge\de q+xy\de y\wedge\de q\;,
\end{multline*}
\begin{multline*}
\omega'=xy\de x\wedge\de p+(-x^2y^2-xy+pq)\de x\wedge\de y-\de x\wedge\de q\\-pq\de p\wedge\de y+\de p\wedge\de q+(-xy-1)\de y\wedge\de q\;,
\end{multline*}
\begin{multline*}
\rho=(-xpy+xy^2+x-q)\de x+(xpy-x+q)\de p\\+(-x^2py^2+x^2y-xpy-xyq+pyq+x-q)\de y-y\de q\;,
\end{multline*}
\begin{multline*}
\rho'=(-xpy+xy^2+x-p+y-q)\de x+(xpy-x+p+q)\de p\\+(-x^2py^2+x^2y-xpy-xyq+pyq)\de y-y\de q\;,
\end{multline*}
\[
\sigma=(p-y)\de x-p\de p+(x-q)\de y\;,\hspace*{6.5cm}
\]
which lead to
\begin{eqnarray*}
\mathcal{I}^1_{\mathcal{D}}&=&-2xy+1\;;\\
\mathcal{I}^2_{\mathcal{D}}&=&2xy-2py-2yq\;;\\
\mathcal{I}^3_{\mathcal{D}}&=&2 x^{2}py^{3} + 2xp^{2}y^{3} - x^{2}y^{4}-2x^{2}y^{2}\\&&\hspace*{3cm}- xpy^{2} + p^{2}y^{2} - xy^{3} + 2xy^{2}q + py^{2}q + y^{3}q- py\;;\\
\mathcal{I}^5_{\mathcal{D}}&=&2py+1\;.
\end{eqnarray*}

The above expressions easily give $xy,py,yq$ as polynomials in $\mathcal{I}^1_{\mathcal{D}},\mathcal{I}^2_{\mathcal{D}},\mathcal{I}^5_{\mathcal{D}}$ and, consequently, $y^2\left(-x^2y^2-xy+yq\right)$ as a polynomial in $\mathcal{I}^1_{\mathcal{D}},\mathcal{I}^2_{\mathcal{D}},\mathcal{I}^3_{\mathcal{D}},\mathcal{I}^5_{\mathcal{D}}$. Then, in the open (and dense) domain $V:=\left\{-x^2y^2-xy+yq\ne 0\right\}$, coordinates $x,y,p,q$ are smooth functions of $\mathcal{I}^1_{\mathcal{D}}, \mathcal{I}^2_{\mathcal{D}}, \mathcal{I}^3_{\mathcal{D}}, \mathcal{I}^5_{\mathcal{D}}$. This, obviously, implies the independence of the latter in~$V$.
But $\mathcal{I}^1_{\mathcal{D}}, \mathcal{I}^2_{\mathcal{D}}, \mathcal{I}^3_{\mathcal{D}}, \mathcal{I}^5_{\mathcal{D}}$ are pullbacks of $\mathcal{I}^1,\mathcal{I}^2,\mathcal{I}^3,\mathcal{I}^5$ through the section of $\operatorname{J}^2(\gamma)$ corresponding to $\mathcal{D}$. Since in the jet-coordinates extending those in \hyperref[BoE]{n.~\ref{BoE}}, $\mathcal{I}^1,\mathcal{I}^2,\mathcal{I}^3,\mathcal{I}^5$ are rational functions, independence even at a single $\theta\in\operatorname{J}^2\left(\gamma\right)$ implies independence over a (Zariski) open and dense subset. Thus we conclude that $\mathcal{I}^1,\mathcal{I}^2,\mathcal{I}^3,\mathcal{I}^5$ are independent~\footnote{We also have $\mathcal{I}^4_{\mathcal{D}}=-2$, $\mathcal{I}^6_{\mathcal{D}}=-4xpy^2-2p^2y^2-4xy+2$. Hence for $\mathcal{D}$ these invariants functionally depend on $\mathcal{I}^1_{\mathcal{D}},\mathcal{I}^5_{\mathcal{D}}$. To prove by hands some other independence results, one may change distribution. For instance, independence of $\mathcal{I}^1,\mathcal{I}^2,\mathcal{I}^3,\mathcal{I}^4$ may be verified by using the distribution
\[
\left\langle\;\partial_p+\partial_y+pq\partial_q\;,	\;\partial_x+xy\partial_p\;\right\rangle\;.
\]
}.
\end{proof}

Consider a $2$-distribution $\mathcal{D}$ and the values of four independent differential invariants, say $\mathcal{I}^1_{\mathcal{D}},\mathcal{I}^4_{\mathcal{D}},\mathcal{I}^5_{\mathcal{D}},\mathcal{I}^2_{\mathcal{D}}$, as a local chart on $M$. Then the components of the projector $P$ in this local chart characterize completely the equivalence class of $\mathcal{D}$. These components can be found as follows. Consider differential forms
\[
\alpha_1:=P^\ast\left(\de\mathcal{I}^1_{\mathcal{D}}\right)\,,\quad\alpha_2:=P^\ast\left(\de\mathcal{I}^4_{\mathcal{D}}\right)\,,\quad \alpha_3:=P^\ast\left(\de\mathcal{I}^5_{\mathcal{D}}\right)\,,\quad\alpha_4:= P^\ast\left(\de\mathcal{I}^2_{\mathcal{D}}\right)\;,
\]
where $P^\ast:\operatorname{\Lambda}^1(M)\to\operatorname{\Lambda}^1(M)$ is the dual of $P:\D{M}\to\D{M}$. These forms are, obviously, invariants of $\mathcal{D}$, and their components in the considered local chart are nothing but the components of the tensor $P$ in this chart.

\section{Second Order Differential Invariants}

All scalar differential invariants constructed in \hyperref[DIDS]{Section~\ref{DIDS}} are, as it is easy to see, of second order. In this section we shall show that invariants $\mathcal{I}^1,\ldots ,\mathcal{I}^7$ form a complete system of second order scalar differential invariants.

First of all we have the following result.
\begin{prop}\label{Compind}
The invariants $\mathcal{I}^1,\ldots ,\mathcal{I}^7$ are independent.
\end{prop}
\begin{proof}
As in the proof of \autoref{Fourindip}, observe that the considered invariants are rational functions in the coordinates introduced in \hyperref[BoE]{n.~\ref{BoE}}. Hence it is sufficient to verify their independence at a suitable single point $\theta\in\operatorname{J}^2\left(\gamma\right)$ only. With this simplification a direct computer check gives the desired result.
\end{proof}

\begin{rem}
\autoref{Fourindip} is obviously a consequence of the above proposition. However, we preferred an independent proof because it can be done by hands.
On the contrary, a by hands proof of independence of $\mathcal{I}^1,\ldots ,\mathcal{I}^7$ would require an unreasonable `spacetime'.
\end{rem}

Let $\gamma$ be as in \hyperref[BoE]{n.~\ref{BoE}} and denote by $r$ the maximal number of second order independent invariants. In order to prove that $r\le 7$ it is sufficient to show that the codimension of generic orbits of a natural action of symplectomorphisms of $\left(M,\Omega\right)$ on $\operatorname{J}^2\left(\gamma\right)$ is at most $7$. To this end, we shall consider natural lifts of Hamiltonian fields on $M$ to $\operatorname{J}^2\left(\gamma\right)$ and generated by them subspaces $H_\theta\subset\operatorname{T}_\theta\left(\operatorname{J}^2\left(\gamma\right)\right)$, for all $\theta\in\operatorname{J}^2\left(\gamma\right)$. Obviously, $r$ is not greater than the codimension $r_\theta$ of $H_\theta$. So, it suffices to find a point $\theta$ for which $r_\theta=7$. By making some simple computer tests, we easily find such $\theta$. In these computations we used \cocoa\ (see \cite{C}). Independently, this check was done with Maple{\tiny$^{\text{\texttrademark}}$} by M.\ Marvan. Thus we have

\begin{prop}\label{Atmost}
There are no more than $7$ independent second order scalar differential invariant of $2$-distributions in~$\left(M,\Omega\right)$.
\end{prop}

\section{Differential Invariants of Symplectic MAEs}\label{DISM}

Since a symplectic MAE $\mathcal{E}$ is identified with the unordered pair of distributions $\left\{\mathcal{D}_{\mathcal{E}},\mathcal{D}_\mathcal{E}'\right\}$, a differential invariant of $\mathcal{D}_{\mathcal{E}}$ (or of $\mathcal{D}_{\mathcal{E}}'$) is a differential invariant of $\mathcal{E}$ if and only if it is invariant with respect to the involution $\mathfrak{I}$. By using invariants $\mathcal{I}^1,\ldots ,\mathcal{I}^7$ of $2$-dimensional distributions it is not difficult to construct from them $\mathfrak{I}$-invariant polynomials by using \autoref{Involution}. One of many possibilities to do that is as follows:
\begin{equation}\label{J}
\begin{array}{l}
\mathcal{J}^1:=\mathcal{I}^1\;,\\
\mathcal{J}^2:=\mathcal{I}^2\;,\\
\mathcal{J}^3:=\mathcal{I}^3{\mathcal{I}^3}'=\mathcal{I}^2\mathcal{I}^3-\left(\mathcal{I}^3\right)^2\\
\mathcal{J}^4:=\mathcal{I}^4\;,\\
\mathcal{J}^5:=\mathcal{I}^5{\mathcal{I}^5}'=\mathcal{I}^4\mathcal{I}^5-\left(\mathcal{I}^5\right)^2\\
\mathcal{J}^6:=\mathcal{I}^6-\mathcal{I}^5\\
\mathcal{J}^7:=\mathcal{I}^7{\mathcal{I}^7}'=-\left(\mathcal{I}^7\right)^2
\end{array}
\end{equation}
These invariants are independent at every $\theta\in\operatorname{J}^2\left(\gamma\right)$ where $\mathcal{I}^1,\ldots,\mathcal{I}^7$ are independent and $\mathcal{I}^3\ne{\mathcal{I}^3}'$, $\mathcal{I}^5\ne{\mathcal{I}^5}'$, $\mathcal{I}^7\ne 0$. Thus, in view of \autoref{Compind}, they are independent invariants for \emph{generic hyperbolic} symplectic MAEs.

This result is interesting in its own, but can easily be extended to the elliptic case. To this end, we notice that an $\Omega$-selfadjoint operator $A:\D{M}\to\D{M}$ is naturally associated with a symplectic MAE $\mathcal{E}$. This operator is a symplectic version of the operator $\mathbf{A}$ described in~\hyperref[MAE]{n.~\ref{MAE}}. Solutions of $\mathcal{E}$ are Lagrangian submanifolds $L\subset M$ such that $\operatorname{T}_xL$ is an invariant subspace of $A_x:\operatorname{T}_x\to\operatorname{T}_x$, $\forall x\in L$. When $\mathcal{E}$ is hyperbolic, then $A=P-P'$ or $A=P'-P$ with $P$, $P'$ being the $\Omega$-orthogonal projectors defined in \hyperref[DIDS]{Section~\ref{DIDS}}. Alternatively, this operator $A$ can be characterized as an $\Omega$-selfadjoint operator such that $A^2=\operatorname{id}$, $A\ne\pm\operatorname{id}$. Similarly, an elliptic (resp., parabolic) MAE is associated with an $\Omega$-selfadjoint operator such that $A^2=-\operatorname{id}$ (resp., $A^2=0$, $A\ne 0$). In particular, for hyperbolic and elliptic equations the operator $A$ is uniquely defined up to the sign. Hence symplectic differential invariants of such an operator $A$ that are invariant with respect to the involution $A\to -A$ are differential invariants of MAE associated with $A$. By this reason, in order to construct symplectic differential invariants for elliptic MAEs it is sufficient to express previously found invariants for hyperbolic MAEs in terms of the operator~$A$. Namely, we have
\begin{lemma}
\label{BhA}
If $\mathcal{D}$ is a non-Lagrangian $2$-distribution on $M$ and $A:\D{M}\to\D{M}$ is such that $\left.A\right|_{\mathcal{D}}=\operatorname{id}$ and $\left.A\right|_{\mathcal{D}'}=-\operatorname{id}$, then
\[
\omega-\omega'=\frac 12A\ins\Omega\;,\qquad\sigma=\frac12\left[\ast\de\left(\omega-\omega'\right)\right]\;,\qquad\rho+\rho'=-\Gamma A\Gamma^{-1}\sigma\;.
\]
\end{lemma}
\begin{proof}
It follows from the obvious relation $P=\frac12\left(\operatorname{id}_{\D{M}}+A\right)$, \autoref{Ins} and \autoref{N5}.
\end{proof}
Forms 
\[
\theta:=\theta_A:=\frac 12A\ins\Omega\,,\;\sigma:=\sigma_A:=\frac14\ast\de\left(A\ins\Omega\right)\,,\;\varrho:=\varrho_A:=-\frac14\Gamma A\Gamma^{-1}\left(\ast\de\left(A\ins\Omega\right)\right)
\]
are differential invariants of the operator $A$. By \autoref{BhA} in the hyperbolic case we have
\[
\omega=\frac12\left(\Omega+\theta\right)\;,\qquad\rho=\frac12\left(\varrho+\sigma\right)\;.
\]
By substituting these relations for $\omega$, $\rho$ in formulas \eqref{J} we find the description of invariants $\mathcal{I}^k$'s and consequently of $\mathcal{J}^k$'s in terms of $\theta$, $\sigma$ and $\varrho$. Since these expressions for $\mathcal{J}^k$'s are invariant with respect to the involution $A\to -A$, they are differential invariants of the associated hyperbolic MAEs. According to the above said they also give differential invariants of elliptic MAEs. However, by some reasons, it is more convenient to use invariants
\begin{gather*}
\tilde{\mathcal{J}}^1_{\mathcal{E}}:=\ast\left(\theta\wedge\de\sigma\right)\;,\\
\tilde{\mathcal{J}}^2_{\mathcal{E}}:=\ast\left(\sigma\wedge\varrho\wedge\de\sigma\right)\;,\quad\tilde{\mathcal{J}}^3_{\mathcal{E}}:=\left[\ast\left(\sigma\wedge\varrho\wedge\de\varrho\right)\right]^2\;,\\
\tilde{\mathcal{J}}^4_{\mathcal{E}}:=\ast\left(\left(\de\sigma\right)^2\right)\;,\quad\tilde{\mathcal{J}}^5_{\mathcal{E}}:=\left[\ast\left(\de\sigma\wedge\de\varrho\right)\right]^2\;,\quad\tilde{\mathcal{J}}^6_{\mathcal{E}}:=\ast\left(\left(\de\varrho\right)^2\right)\;,\\
\tilde{\mathcal{J}}^7_{\mathcal{E}}:=\left[\ast\left(\sigma\,\wedge\,\de\varrho\,\wedge\,\ast\left(\sigma\wedge\de\sigma\right)\right)\right]^2\;,\quad\tilde{\mathcal{J}}^8_{\mathcal{E}}:=\ast\left(\sigma\,\wedge\,\de\varrho\,\wedge\,\ast\left(\varrho\wedge\de\sigma\right)\right)\;,\\
\quad\tilde{\mathcal{J}}^9_{\mathcal{E}}:=\ast\left(\sigma\,\wedge\,\de\sigma\,\wedge\,\ast\left(\varrho\wedge\de\varrho\right)\right)\;.
\end{gather*}

This way we get common differential invariants $\tilde{\mathcal{J}}^1$, $\ldots$, $\tilde{\mathcal{J}}^9$ for elliptic and hyperbolic symplectic MAEs.

\begin{thm}
The differential invariants
\[
\tilde{\mathcal{J}}^1\;,\quad\tilde{\mathcal{J}}^2\;,\quad\tilde{\mathcal{J}}^3\;,\quad\tilde{\mathcal{J}}^4\;,\quad\tilde{\mathcal{J}}^5\;,\quad\tilde{\mathcal{J}}^6\;,\quad\tilde{\mathcal{J}}^7
\]
are independent, and seven is the maximum possible order for a system of independent invariants for symplectic MAEs.
\end{thm}
\begin{proof}
Identical to that of \autoref{Atmost} and \autoref{Compind}. Alternatively, the assertion concerning upper bound may be obtained as the ``analytical continuation'' of the hyperbolic part. Indeed, as it is easy to see, the lifting of Hamiltonian vector fields to $\operatorname{J}^2(\pi)$ is described by polynomial functions in the local chart in $\operatorname{J}^2(\pi)$ that is a natural extension of the chart $\left(x,p,y,q,v^1,v^2,v^3,v^4\right)$ introduced in \hyperref[BoE]{n.~\ref{BoE}}.
\end{proof}

Differential invariants of a contact MAE with a fixed symmetry $X$ can easily be obtained from the corresponding symplectic equations. Indeed, if $\mathcal{I}_{\mathcal{E}}$ is a differential invariant of such a contact equation, then $X\left(\mathcal{I}_{\mathcal{E}}\right)=0$. This means that (locally) $\mathcal{I}_{\mathcal{E}}=\pi_M^\ast\left(\mathcal{J}_{\mathcal{E}_{\textrm{sp}}}\right)$, with $\mathcal{J}_{\mathcal{E}_{\textrm{sp}}}$ being a differential invariant of the corresponding symplectic equation. If $X$ is multiplied by a constant factor, the symplectic structure on $M$ corresponding to $X$ does the same. So, differential invariants of contact MAE with a fixed one-dimensional algebra of symmetries are those differential invariants of symplectic MAEs that do not change when the underlying symplectic structure is multiplied by a constant factor. It is easy to see that the passage from $\Omega$ to $c\Omega$ transforms basic differential invariants $\tilde{\mathcal{J}}^1$, $\ldots$, $\tilde{\mathcal{J}}^7$ as follows:
\[
c^{-1}\tilde{\mathcal{J}}^1\;,\quad c^{-2}\tilde{\mathcal{J}}^2\;,\quad c^{-4}\tilde{\mathcal{J}}^3\;,\quad c^{-2}\tilde{\mathcal{J}}^4\;,\quad c^{-4}{\mathcal{J}}^5\;,\quad c^{-2}\tilde{\mathcal{J}}^6\;,\quad c^{-6}\tilde{\mathcal{J}}^7\;.
\]
Now, by dividing these invariants by the appropriate power of the first one, we obtain contact differential invariants
\[
\frac{\tilde{\mathcal{J}}^2}{\left(\tilde{\mathcal{J}}^1\right)^2}\;,\quad\frac{\tilde{\mathcal{J}}^3}{\left(\tilde{\mathcal{J}}^1\right)^4}\;,\quad\frac{\tilde{\mathcal{J}}^4}{\left(\tilde{\mathcal{J}}^1\right)^2}\;,\quad\frac{\tilde{\mathcal{J}}^5}{\left(\tilde{\mathcal{J}}^1\right)^4}\;,\quad\frac{\tilde{\mathcal{J}}^6}{\left(\tilde{\mathcal{J}}^1\right)^2}\;,\quad\frac{\tilde{\mathcal{J}}^7}{\left(\tilde{\mathcal{J}}^1\right)^6}\;,
\]
for contact MAEs with a fixed one-dimensional algebra of symmetries.

\section{Higher Order Invariants and Symmetries}

Invariant vector fields $Z$, $Z'$ of the distribution $\mathcal{D}$ (see \hyperref[DIDS]{Section~\ref{DIDS}}) are of the first jet order. It is not difficult to construct second order invariant vector fields for $\mathcal{D}$. Namely, such are
\begin{multline*}
Z_{00}=\Gamma^{-1}\ast\left(\rho\wedge\de\rho\right)\;,\quad Z_{01}=\Gamma^{-1}\ast\left(\rho\wedge\de\rho'\right)\;,\\ Z_{10}=\Gamma^{-1}\ast\left(\rho'\wedge\de\rho\right)\;,\quad Z_{11}=\Gamma^{-1}\ast\left(\rho'\wedge\de\rho'\right)\;.
\end{multline*}
An alternative definition of fields $Z_{ij}$ is 
\begin{multline*}
Z_{00}\ins\frac12\Omega^2=\rho\wedge\de\rho\;,\quad Z_{01}\ins\frac12\Omega^2=\rho\wedge\de\rho'\;,\\ Z_{10}\ins\frac12\Omega^2=\rho'\wedge\de\rho\;,\quad Z_{11}\ins\frac12\Omega^2=\rho'\wedge\de\rho'\;.
\end{multline*}

\begin{prop}\label{NR2}
For a generic distribution $\mathcal{D}$, the invariants $Z_{00}$, $Z_{01}$, $Z_{10}$, $Z_{11}$ are linearly independent fields.
\end{prop}
\begin{proof}
It suffices to find a distribution for which these fields are independent. For instance, a such one is that in the proof of \autoref{Fourindip}.
\end{proof}
 
According to this proposition, four invariant vector fields $Z_{ij}$ form an invariant e-structure whose invariants, scalar or not, are differential invariants of $\mathcal{D}$. Moreover, one can construct various invariant e-structures as combinations of invariant vector fields $Z$, $Z'$ and $Z_{ij}$. For instance, the e-structure considered in \cite[Sect.~6, Theorem~4]{K} is $\left(-2Z,-2Z',1/\left(2\mathcal{I}^3_{\mathcal{D}}-2\mathcal{I}^2_{\mathcal{D}}\right)P\left(\left[Z,Z'\right]\right),1/\left(2\mathcal{I}^3_{\mathcal{D}}\right)P'\left(\left[Z,Z'\right]\right)\right)$. It should be stressed that second order SDIs of SMAEs derived from this e-structure come from the commutator $\left[Z, Z'\right]$ (cf.~\cite[p.~392]{KLR}), while other commutators of these invariant vector fields give SDIs of $3$-rd order. \footnote{Vector fields composing the e-structure considered in \cite[p.~435]{KLR} involve fields $Z$, $Z'$, $Z_{ij}$, the operator $A$ and other second order scalar differential invariants and hence are rather complicated.}

Now we have at our disposal four invariant differential forms, namely, $\omega$, $\omega'$, $\rho$, $\rho'$ and six invariant vector fields $Z$, $Z'$, $Z_{ij}$. By applying to them standard operations of tensor analysis we easily obtain numerous differential invariants of higher order. In particular, by successively applying these vector fields to invariants $\mathcal{I}^k_{\mathcal{D}}$'s we find scalar differential invariants of higher than two order.

Since the symplectic form $\Omega$ is a differential invariant for $\mathcal{D}$, the Poisson bracket of two scalar differential invariants is a scalar differential invariant as well. For instance, $\left\{\mathcal{I}^k_{\mathcal{D}},\mathcal{I}^l_{\mathcal{D}}\right\}$ is a third order differential invariant of $\mathcal{D}$.

Recall that a classical (infinitesimal) symmetry of a PDE $\mathcal{E}\subset\operatorname{J}^k$ is a contact vector field whose natural lift to $\operatorname{J}^k$ is tangent to $\mathcal{E}$. In our context this translates to be a Hamiltonian field that leave invariant the distribution~$\mathcal{D}$. Obviously, the value of a scalar differential invariant is constant along a trajectory of a symmetry. This implies that if generic orbits of the symmetry algebra of $\mathcal{D}$ is of dimension~$l$, then the number of independent differential invariants of $\mathcal{D}$ is not greater than $4-l$. In particular, a MAE does not admit nontrivial infinitesimal symmetries if it possesses four independent invariants. Moreover, if a Hamiltonian vector field $X_f$ is a symmetry of a symplectic MAE $\mathcal{E}$, then $\left\{f,\mathcal{I}_{\mathcal{E}}\right\}=0$ for any scalar differential invariant $\mathcal{I}$. This observation is very useful in practical search of symmetries for concrete MAEs.

\section{An Application}

Invariants $\tilde{\mathcal{J}}^1,\ldots ,\tilde{\mathcal{J}}^7$ are independent for generic symplectic MAEs, nevertheless, they and related invariant differential forms and vector fields are useful for nongeneric equations as well.  In this section we illustrate this point by applying the previously developed machinery to hyperbolic equations of the form
\begin{equation}
\label{Special}
u_{xy}+D=0\;,\qquad D=D\left(x,y,u_x,u_y\right)\;.
\end{equation}
In particular, we shall give a solution of the linearization problem, i.e., when a symplectic hyperbolic MAE is equivalent to a linear one. Distributions $\mathcal{D}$ and $\mathcal{D}'$ associated with \eqref{Special} are
\begin{equation}
\label{SpecialD}
\mathcal{D}=\langle\;\partial_p\,,\,\partial_x-D\partial_q\;\rangle\,,\quad\mathcal{D}'=\langle\;\partial_q\,,\,\partial_y-D\partial_p\;\rangle
\end{equation}
and hence
\[
Z=-D_p\partial_q\;,\quad \rho=-D_p\de y\;,\quad Z'=-D_q\partial_p\;,\quad \rho'=-D_q\de x\;.
\]
The distributions $\mathcal{D}_{(1)}$ and $\mathcal{D}'_{(1)}$ are integrable \footnote{$\mathcal{D}_{(1)}$ denotes the distribution generated by $\mathcal{D}$ and $\left[\mathcal{D},\mathcal{D}\right]$.}, and, if $\mathcal{D}$ and $\mathcal{D}'$ are not integrable, then $\mathcal{D}_{(1)}=\left\{\de y=0\right\}=\langle Z\rangle^\perp$ and $\mathcal{D}'_{(1)}=\left\{\de x=0\right\}=\langle Z'\rangle^\perp$. The inverse assertion is also true.

\begin{prop}\label{Specialprop}
A non-Lagrangian $2$-distribution $\mathcal{D}$ is associated with a symplectic equation equivalent to \eqref{Special} if and only if distributions $\mathcal{D}_{(1)}$ and $\mathcal{D}'_{(1)}$ are integrable.
\end{prop}
\begin{proof}
Assume that $\mathcal{D}$ and $\mathcal{D}'$ are not integrable, i.e., that $\mathcal{D}_{(1)}$ and $\mathcal{D}'_{(1)}$ are $3$-dimensional. Therefore, there are (locally) functions $x,y\in\operatorname{C}^\infty(M)$ such that $\mathcal{D}_{(1)}=\left\{\de y=0\right\}$ and $\mathcal{D}'_{(1)}=\left\{\de x=0\right\}$, or, equivalently, $\mathcal{D}_{(1)}=\langle X_y\rangle^\perp$ and $\mathcal{D}'_{(1)}=\langle X_x\rangle^\perp$, where $X_H$ stands for the Hamiltonian vector field with the Hamiltonian $H\in\operatorname{C}^{\infty}(M)$. On the other hand, $Z$ is $\Omega$-orthogonal to $\mathcal{D}$ and belongs to $\mathcal{D}_{(1)}$. So, $\mathcal{D}_{(1)}=\langle Z\rangle^\perp$. This implies proportionality of $Z$ and $X_y$ and we put $Z=-\lambda X_y$, $\lambda\in\operatorname{C}^\infty(M)$. Similarly we find that $Z'=-\lambda' X_x$. Note that, by the assumption, $Z\ne 0$, $Z'\ne 0$.

Since $\Omega\left(Z,Z'\right)=\lambda\lambda'\Omega(X_y,X_x)=\lambda\lambda'\left\{x,y\right\}$, the $\Omega$-orthogonality of $Z$ and $Z'$ implies that $\left\{x,y\right\}=0$. Hence there exists a canonical chart of the form $\left(x,p,y,q\right)$, i.e., $\Omega=\de p\wedge\de x+\de q\wedge\de y$. In such a chart, $Z=-\lambda\partial_q$, $Z'=-\lambda'\partial_p$ and hence $\rho=-\lambda\de y$, $\rho'=-\lambda'\de x$. Since $\mathcal{D}_{(1)}=\left\{\de y=0\right\}$ and $Z'\in\mathcal{D}$, the distribution $\mathcal{D}$ is generated by $\partial_p$ and a vector field of the form $\alpha\partial_q+\beta\partial_x$. Since $\mathcal{D}$ is not Lagrangian, $\beta\ne 0$, and $\mathcal{D}=\langle\partial_p,\partial_x-D\partial_q\rangle$ with $D=-\alpha/\beta$. Similarly, we find that $\mathcal{D}'=\langle\partial_q,\partial_y-D'\partial_p\rangle$. Finally, $\Omega$-orthogonality of $\partial_x-D\partial_q$ and $\partial_y-D'\partial_p$ implies $D=D'$. This shows that $\mathcal{D}$ and $\mathcal{D}'$ are of the form \eqref{SpecialD}. This proves the assertion for nonintegrable $\mathcal{D}$ and $\mathcal{D}'$.

Now assume that $\mathcal{D}$ is not integrable and $\mathcal{D}'$ is integrable. As above we see that $\mathcal{D}_{(1)}=\left\{\de y=0\right\}=\langle X_y\rangle^\perp$ and $\mathcal{D}'=\left\{\de x=0, \de f=0\right\}$ for some functions $x,y,f\in\operatorname{C}^\infty(M)$. Then $\mathcal{D}=\langle X_x, X_f\rangle$. The inclusion $\mathcal{D}\subset\mathcal{D}_{(1)}$ implies $\de y\left(X_x\right)=\de y\left(X_f\right)=0$, or, equivalently, $\left\{x,y\right\}=\left\{f,y\right\}=0$. Since $\left\{x,y\right\}=0$, a canonical chart of the form $(x,p,y,q)$ exists and $\left\{f,y\right\}=0\iff f_q=0$. So,
\begin{multline*}
\mathcal{D}=\langle X_x,X_f\rangle=\langle X_x,f_xX_x+f_yX_y+f_pX_p\rangle\\=\langle X_x,f_yX_y+f_pX_p\rangle=\langle\partial_p,f_y\partial_q-f_p\partial_x\rangle\;.
\end{multline*}
Notice that $f_p\ne 0$, otherwise, $\mathcal{D}$ would be Lagrangian. So, $\mathcal{D}=\langle\partial_p,\partial_x-D\partial_q\rangle$ with $D=f_y/f_p$. The distribution $\langle\partial_q,\partial_y-D\partial_p\rangle$ is, obviously, $\Omega$-orthogonal to $\mathcal{D}$ and as such coincides with $\mathcal{D}'$. Hence in the considered case $\mathcal{D}$ and $\mathcal{D}'$ have the form~\eqref{SpecialD} with peculiarity that $D_q=0$.

Finally, if $\mathcal{D}$ and $\mathcal{D}'$ are integrable, then $\mathcal{D}=\left\{\de f=0,\de g=0\right\}$ and $\mathcal{D}'=\left\{\de f'=0,\de g'=0\right\}$. On the other hand, the distribution $\langle X_f,X_g\rangle$ is orthogonal to $\mathcal{D}$ and, hence, coincides with $\mathcal{D}'$. By this reason $X_f\ins\de f'=X_f\ins\de g'=0$, or, equivalently, $\left\{f,f'\right\}=\left\{f,g'\right\}=0$. Similarly, $\left\{g,f'\right\}=\left\{g,g'\right\}=0$. Moreover, integrability of $\mathcal{D}'$ implies $X_{\left\{f,g\right\}}=\left[X_f,X_g\right]\in\mathcal{D}'\iff X_{\left\{f,g\right\}}=\alpha X_f+\beta X_g\iff\de\left\{f,g\right\}=\alpha\de f+\beta\de g$. The last relation shows that $\left\{f,g\right\}$ is a function of $f$ and~$g$. Also, note that $\left(\alpha,\beta\right)\ne (0,0)$ since, otherwise, $\mathcal{D}$ would be Lagrangian. So, the $\operatorname{C}^\infty$--closed subalgebra of $\operatorname{C}^\infty(M)$ generated by $f$ and $g$ is a Poisson subalgebra with nontrivial bracket and as such admits a canonical chart $(y,q)$, $y=y(f,g)$, $q=q(f,g)$. Similarly, one can construct functions $x=x(f',g')$, $p=p(f',g')$ with $\left\{x,p\right\}=1$. Then $(x,p,y,q)$ is a canonical chart for $\Omega$ and $\mathcal{D}=\langle\partial_p,\partial_x\rangle$, $\mathcal{D}'=\langle\partial_q,\partial_y\rangle$. In other words, the corresponding to $\mathcal{D}$ equation is~$u_{xy}=0$.
\end{proof}

We shall call an equation \eqref{Special} \emph{generic} if distributions $\mathcal{D}$ and $\mathcal{D}'$ are both nonintegrable. In this case $\mathcal{D}_{(1)}=\left\{\de y=0\right\}$ and $\mathcal{D}'_{(1)}=\left\{\de x=0\right\}$, with $x,y$ uniquely defined up to a transformation $(x,y)\mapsto\left(\overline{x}=\overline{x}(x),\overline{y}=\overline{y}(y)\right)$. As it is easy to see, the transformation of corresponding canonical charts is \begin{equation}\label{Tcc}(x,p,y,q)\mapsto\left(\;\overline{x}=\overline{x}(x)\,,\,\overline{p}=\frac{1}{\de\overline{x}/\de x}\left(p+\varphi_x\right)\,,\,\overline{y}=\overline{y}(y)\,,\,\overline{q}=\frac{1}{\de\overline{y}/\de y}\left(q+\varphi_y\right)\right)\end{equation} with $\varphi=\varphi(x,y)$ being an arbitrary function. The Lie algebra associated with the pseudo-group~\eqref{Tcc} is formed by Hamiltonian vector fields \[X_{a(x)p+b(y)q+\psi(x,y)}\] where $a(x)$, $b(y)$, $\psi(x,y)$ are arbitrary smooth functions.

For a distribution~\eqref{SpecialD} we have the following obvious relations
\[
\mathcal{L}_{Z}(\rho)=\mathcal{I}^1_{\mathcal{D}}\,\rho\;,\quad\mathcal{L}_{Z'}(\rho')=\mathcal{I}^1_{\mathcal{D}}\,\rho'\,,
\]
\[
\mathcal{L}_{Z}(\rho')=\mathcal{I}^{12}_{\mathcal{D}}\,\rho'\;,\quad\mathcal{L}_{Z'}(\rho)=\mathcal{I}^{21}_{\mathcal{D}}\,\rho
\]
with
\begin{equation}
\label{12D}
\mathcal{I}^{12}_{\mathcal{D}}:=-D_p\frac{D_{qq}}{D_q}\;,\quad\mathcal{I}^{21}_{\mathcal{D}}:=-D_q\frac{D_{pp}}{D_p}
\end{equation}
for generic distributions $\mathcal{D}$ and $\mathcal{D}'$.

So, $\mathcal{I}^{12}_{\mathcal{D}}$ and $\mathcal{I}^{21}_{\mathcal{D}}$ are differential invariants of distributions of the form~\eqref{SpecialD}. As it is easy to see, differential invariants $\mathcal{I}^1,\ldots ,\mathcal{I}^7$ for these distributions are:
\begin{gather}
\label{I1eD}\mathcal{I}^1_{\mathcal{D}}=-D_{pq}\;,\\
\label{I45eD}\mathcal{I}^4_{\mathcal{D}}=-2D_{pq}^2+2D_{pp}D_{qq}\;,\quad\mathcal{I}^5_{\mathcal{D}}=\frac12\mathcal{I}^4_{\mathcal{D}}=-D_{pq}^2+D_{pp}D_{qq}\;,\\
\label{AID}\mathcal{I}^2_{\mathcal{D}}=\mathcal{I}^3_{\mathcal{D}}=\mathcal{I}^6_{\mathcal{D}}=\mathcal{I}^7_{\mathcal{D}}=0\;.
\end{gather}
This shows that $\mathcal{I}^{12}_{\mathcal{D}}$ and $\mathcal{I}^{21}_{\mathcal{D}}$ can not be expressed in terms of the invariants $\mathcal{I}^k$'s. In other words, they are \emph{special} differential invariants, i.e., invariants for the special class of distributions $\mathcal{D}$ considered in this section, i.e., for which $\mathcal{D}_{(1)}$, ${\mathcal{D}'}_{(1)}$ are integrable.

A simple application of these invariants is that they completely characterize hyperbolic symplectic linear equations, i.e., equations of the form
\begin{equation}
\label{HSLE}
u_{xy}+\alpha(x,y)u_x+\beta(x,y)u_y+\gamma(x,y)=0\;.
\end{equation}

\begin{prop}\label{L1}
A hyperbolic symplectic MAE is symplectic equivalent to an equation~\eqref{HSLE} if and only if $\mathcal{D}_{(1)}$ and $\mathcal{D}'_{(1)}$ are integrable and either
\begin{itemize}
\item $Z\ne  0$, $Z'\ne0$ and $\mathcal{I}^1_{\mathcal{D}}=\mathcal{I}^{12}_{\mathcal{D}}=\mathcal{I}^{21}_{\mathcal{D}}=0$, or
\item $Z\ne 0$, $Z'=0$ (resp., $Z=0$, $Z'\ne 0$) and $\rho$ (resp., $\rho'$) admits a Hamiltonian characteristic belonging to $\mathcal{D}$ (resp., to $\mathcal{D}'$), or
\item $Z=Z'=0$.
\end{itemize}
\end{prop}
\begin{proof}
Integrability of $\mathcal{D}_{(1)}$ and $\mathcal{D}'_{(1)}$ allows one to bring the considered equation to the form~\eqref{Special} (\autoref{Specialprop}).

In the case $Z\ne  0$, $Z'\ne0$, vanishing of $\mathcal{I}^{12}_{\mathcal{D}}$ and $\mathcal{I}^{21}_{\mathcal{D}}$ gives $D_{pp}=D_{qq}=0$ (see~\eqref{12D}). Moreover,  by~\eqref{I1eD}, $D_{pq}=0$, and hence $D(x,p,y,q)$ is linear in $p,q$.

When only one of fields $Z$, $Z'$, say $Z'$, vanishes, the additional hypothesis gives a smooth function $\phi$ such that $X_\phi\ne 0$ belongs to $\mathcal{D}$ and is characteristic for~$\rho$. Since $X_\phi$ is $\Omega$-orthogonal to $\mathcal{D}'$, we have $\mathcal{D}'\subset\{\de\phi=0\}$. Looking at the corresponding case in the proof of \autoref{Specialprop}, we can assume that $\phi=x$. Therefore $0=\mathcal{L}_{X_x}(\rho)=-D_{pp}\de y$, i.e., $D_{pp}=0$. This condition, together with $D_q=0$ (which is due to $Z'=0$), again implies linearity of $D$ with respect to $p$, $q$.

In the case $Z=Z'=0$ the equation is equivalent to $u_{xy}=0$ (see the end of the proof of \autoref{Specialprop}).
\end{proof}

As it is easy to see, all invariants $\mathcal{I}^1,\ldots ,\mathcal{I}^7$ vanish for the distributions associated with a symplectic linear equation \eqref{HSLE}. The inverse is not, however, true.

\begin{expl}
For the distribution
\[
\mathcal{D}=\langle\;\partial_q\,,\,-q\partial_x+yq\partial_p+\partial_y\;\rangle\;,
\]
corresponding to the quasilinear equation
\[
u_yu_{xx}-u_{xy}+yu_y=0\;,
\]
we have
\[
\mathcal{D}'=\langle\;\partial_p+q\partial_q\,,\,\partial_x-y\partial_p\;\rangle\;,
\]
\[
Z=-\partial_x+y\partial_p\;,\quad\rho=y\de x+\de p\;,\quad Z'=0\;,\quad\rho'=0\;.
\]
All invariants $\mathcal{I}^1_{\mathcal{D}},\ldots ,\mathcal{I}^7_{\mathcal{D}}$ vanish since $\rho'=0$ and $\de\rho\wedge\de\rho=0$. On the other hand, since $\mathcal{D}_{(1)}$ is not integrable, this equation can not be brought to the form \eqref{SpecialD} (\autoref{Specialprop}), and hence to the form~\eqref{HSLE}.
\end{expl}
\begin{expl}
The distribution
\[
\mathcal{D}=\langle\;\partial_p\,,\,\partial_x-(p^2+x)\partial_q\;\rangle
\]
is of type \eqref{SpecialD}, with $D=p^2+x$. In this case we have
\[
\mathcal{D}'=\langle\;\partial_q\,,\,\partial_y-(p^2+x)\partial_p\;\rangle\;,\quad Z\ne 0\;,\quad Z'=0\;.
\]
Relations~\eqref{I45eD},~\eqref{AID} show that all invariants $\mathcal{I}^1_{\mathcal{D}},\ldots ,\mathcal{I}^7_{\mathcal{D}}$ vanish as for linear equations~\eqref{HSLE}. Nevertheless, the corresponding to $\mathcal{D}$ equation
\[
u_{xy}+u_x^2+x=0
\]
is not symplectic equivalent to~\eqref{HSLE}. Indeed, the invariant distribution $\ker\de\rho+\mathcal{D}'=\left\langle\;\partial_x\,,\,\partial_q\,,\,\partial_y-(p^2+x)\partial_p\;\right\rangle$ is $3$-dimensional and not integrable, while the similar distribution for \eqref{HSLE} is integrable.
\end{expl}

\begin{expl}
The distribution
\[
\mathcal{D}=\langle\;\partial_p\,,\,\partial_x-(p^2+q)\partial_q\;\rangle
\]
is of type \eqref{SpecialD} with $D=p^2+q$ and
\[
\mathcal{D}'=\langle\;\partial_q\,,\,\partial_y-(p^2+q)\partial_p\;\rangle\;.
\]
In this case $Z\ne 0$, $Z'\ne 0$ and all invariants $\mathcal{I}^1_{\mathcal{D}},\ldots ,\mathcal{I}^7_{\mathcal{D}}$ vanish according to~\eqref{I1eD}, \eqref{I45eD} and~\eqref{AID}. The corresponding to $\mathcal{D}$ equation
\[
u_{xy}+u_x^2+u_y=0
\]
is not symplectic equivalent to~\eqref{HSLE} because $\mathcal{I}^{21}\ne 0$ (see \autoref{L1}).
\end{expl}

More generally, a generic hyperbolic linear equation
\begin{equation}
\label{GLE}
u_{xy}+\alpha(x,y)u_x+\beta(x,y)u_y+\gamma(x,y)u+\delta(x,y)=0
\end{equation}
may be viewed as a symplectic one. For instance, if $\varphi(x,y)$ is a solution of \eqref{GLE}, then the substitution $u=\exp(v)+\varphi$ brings \eqref{GLE} to the form
\begin{equation}
\label{GLER}
v_{xy}+v_xv_y+\alpha(x,y)v_x+\beta(x,y)v_y+\gamma(x,y)=0\;.
\end{equation}
This reduction of \eqref{GLE} to a symplectic form \eqref{GLER} corresponds to $1$-parametric symmetry group $u=(1-\lambda)u+\lambda\varphi$, $\lambda\in\mathbb{R}$, of \eqref{GLE}, or, equivalently, to the infinitesimal contact symmetry
\[
\left(\varphi-u\right)\partial_u+\left(\varphi_x-p\right)\partial_p+\left(\varphi_y-q\right)\partial_q\;.
\]

\begin{prop}\label{NR}
A symplectic MAE is symplectic equivalent to an equation \eqref{GLER} if and only if distributions $\mathcal{D}_{(1)}$ and $\mathcal{D}'_{(1)}$ are integrable and $\mathcal{I}^1_{\mathcal{D}}=-1$, $\mathcal{I}^{12}_{\mathcal{D}}=\mathcal{I}^{21}_{\mathcal{D}}=0$.
\end{prop}
\begin{proof}
Since $\mathcal{D}_{(1)}$ and $\mathcal{D}'_{(1)}$ are integrable, then, by \autoref{Specialprop}, the considered equation is of type~\eqref{Special}. Moreover, it is a generic equation of type~\eqref{Special}, since $\mathcal{I}^1_{\mathcal{D}}\ne 0$ easily implies that $Z\ne 0$ and $Z'\ne 0$. Now, the same arguments as in the proof of the first case of \autoref{L1} prove the linearity.
\end{proof}

For the distribution associated with \eqref{GLER} we have $\mathcal{I}^1_{\mathcal{D}}=\mathcal{I}^5_{\mathcal{D}}=-1$,  $\mathcal{I}^4_{\mathcal{D}}=-2$, while $\mathcal{I}^2,\mathcal{I}^3,\mathcal{I}^6,\mathcal{I}^7$ vanish. The distribution $\langle\;\partial_p\,,\,\partial_x-(pq+p^2+q)\partial_q\;\rangle$ associated with the equation
\begin{equation}\label{NSE}
u_{xy}+u_xu_y+u_x^2+u_y=0
\end{equation}
has the same values of invariants $\mathcal{I}^1,\ldots ,\mathcal{I}^7$. However, since $\mathcal{I}^{21}$ for this distribution is different from zero, Equation~\eqref{NSE} is not symplectic equivalent to \eqref{GLER}.

Similar results can easily be obtained for the equation
\[
u_{xx}+u_{yy}+D=0,\qquad D=D(x,y,u_x,u_y)\;,
\]
which is an elliptic analogue of equation~\eqref{Special}. To this end, it suffices to use forms $\sigma$ and $\varrho$ of  \hyperref[DISM]{Section~\ref{DISM}}, which are elliptic substitutes of $\rho$ and $\rho'$.

\section{\texorpdfstring{Classes of Forms $\rho$ and $\rho'$}{Classes of Forms}}

Recall that the \emph{class} of a differential $1$-form is the number of independent variables figuring in its normal (Darboux) form. Denote by $r$ and $r'$ classes of differential forms $\rho$ and $\rho'$, respectively. We shall show that $\rho$ and $\rho'$ can be of any possible classes from $0$ to $4$. First of all, it is easy to see that all pairs $\left(r,r'\right)$ with $0\le r,r'\le 2$ are realized by distributions of the form \eqref{SpecialD}. In \autoref{Table} we indicate distributions which realize all remaining pairs $\left(r,r'\right)$. It is worth noticing that $r=4$ (resp., $r'=4$) if and only if ${\mathcal{I}^6_\mathcal{D}}\ne 0$ (resp., $\mathcal{I}^6_{\mathcal{D}'}\ne 0$). Also, $r\le 2$ (resp., $r'\le 2$) if and only if $\mathcal{D}_{(1)}$ is integrable (resp., $\mathcal{D}_{(1)}'$ is integrable).

\begin{table}
\caption{Examples for various classes $\left(r,r'\right)$.\label{Table}}
\medskip
\hskip-4cm\begin{minipage}{\textwidth}\scriptsize
\begin{tabular}{l|ccc}\hline
\multicolumn1l{$\left(r,r'\right)$} & $\mathcal{D}$ & $\mathcal{D}'$ & $\mathcal{E}$\\\hline\hline
$(0,3)$ & \distr{\partial_p+q\partial_q}{\partial_x+yq\partial_q} & \distr{\partial_q}{-q\partial_x+yq\partial_p+\partial_y} & $-u_yu_{xx}+u_{xy}-yu_y=0$\\
$(0,4)$ & \distr{-\partial_x+y\partial_p}{q\partial_p-\partial_q} & \distr{\partial_q}{\partial_x-y\partial_p+q\partial_y} & $u_{xx}+u_yu_{xy}+y=0$\\
$(1,3)$ & \distr{\partial_p+(q+1)\partial_q}{\partial_x+yq\partial_q} & \distr{\partial_q}{-(q+1)\partial_x+yq\partial_p+\partial_y} & $-(u_y+1)u_{xx}+u_{xy}-yu_y=0$\\
$(1,4)$ & \distr{\partial_p+f\footnote{\scriptsize$f:=y+q^2\exp(-x)$}\partial_q}{\partial_x+q\partial_q} & \distr{\partial_q}{-f\partial_x+q\partial_p+\partial_y} & $-F\footnote{\scriptsize$F:=y+u_y^2\exp(-x)$}u_{xx}+u_{xy}-u_y=0$\\
$(2,3)$ & \distr{\partial_p+p\partial_q}{\partial_x+q\partial_q}  & \distr{\partial_q}{-p\partial_x+q\partial_p+\partial_y} & $-u_xu_{xx}+u_{xy}-u_y=0$\\
$(2,4)$ & \distr{\partial_p+p\partial_q}{\partial_x+q^2\partial_q} & \distr{\partial_q}{-p\partial_x+q^2\partial_p+\partial_y} & $-u_xu_{xx}+u_{xy}-u_y^2=0$ \\
$(3,3)$ & \distr{\partial_p+\partial_y+p^2\partial_q}{\partial_x+y\partial_p} & \distr{\partial_x+y\partial_p+\partial_q}{\partial_y+p^2\partial_q} &$u_{xx}u_{yy}-u_{xy}^2-u_x^2u_{xx}+u_{xy}-yu_{yy}+yu_x^2=0$\\
$(3,4)$ & \distr{\partial_p+\partial_y+pq\partial_q}{\partial_x+y\partial_p} & \distr{\partial_x+y\partial_p+\partial_q}{\partial_y+pq\partial_q} & $u_{xx}u_{yy}-u_{xy}^2-u_xu_yu_{xx}+u_{xy}-yu_{yy}+yu_xu_y=0$\\
$(4,4)$ & \distr{\partial_p+\partial_y+pq\partial_q}{\partial_x+xy\partial_p} & \distr{\partial_x+xy\partial_p+\partial_q}{\partial_y+pq\partial_q} & $u_{xx}u_{yy}-u_{xy}^2-u_xu_yu_{xx}+u_{xy}-xyu_{yy}+xyu_xu_y=0$ \\\hline
\end{tabular}
\end{minipage}
\end{table}
\normalsize

\section*{Acknowledgements}

The first author was partially supported by the Italian Ministry of Education, University and Research (MIUR).

We thank Prof.\ Michal Marvan for providing us with a cross-checking on the upper bound for the number of independent invariants.

\bibliographystyle{plain}

\begin{thebibliography}{10}

\bibitem{ALV}
D.~V. Alekseevskij, V.~V. Lychagin, and A.~M. Vinogradov.
\newblock Basic ideas and concepts of differential geometry.
\newblock In {\em Geometry, {I}}, volume~28 of {\em Encyclopaedia Math. Sci.},
  pages 1--264. Springer, Berlin, 1991.

\bibitem{KV}
A.~V. Bocharov, V.~N. Chetverikov, S.~V. Duzhin, N.~G. Khor{\cprime}kova, I.~S.
  Krasil{\cprime}shchik, A.~V. Samokhin, Yu.~N. Torkhov, A.~M. Verbovetsky, and
  A.~M. Vinogradov.
\newblock {\em Symmetries and conservation laws for differential equations of
  mathematical physics}, volume 182 of {\em Translations of Mathematical
  Monographs}.
\newblock American Mathematical Society, Providence, RI, 1999.
\newblock Edited and with a preface by Krasil{\cprime}shchik and Vinogradov,
  Translated from the 1997 Russian original by Verbovetsky [A. M.
  Verbovetski{\u\i}] and Krasil{\cprime}shchik.

\bibitem{CV}
Diego Catalano~Ferraioli and Alexandre~M. Vinogradov.
\newblock Differential invariants of generic parabolic monge-ampere equations.
\newblock Preprint arXiv:0811.3947, available at \/ {\tt
  http://arxiv.org/abs/0811.3947}.

\bibitem{C}
{CoCoA}Team.
\newblock {{\hbox{\rm C\kern-.13em o\kern-.07em C\kern-.13em o\kern-.15em A}}}:
  a system for doing {C}omputations in {C}ommutative {A}lgebra.
\newblock Available at \/ {\tt http://cocoa.dima.unige.it}.

\bibitem{K}
Boris Kruglikov.
\newblock Classification of {M}onge-{A}mp\`ere equations with two variables.
\newblock In {\em Geometry and topology of caustics---{CAUSTICS} '98
  ({W}arsaw)}, volume~50 of {\em Banach Center Publ.}, pages 179--194. Polish
  Acad. Sci., Warsaw, 1999.

\bibitem{KLR}
Alexei Kushner, Valentin Lychagin, and Vladimir Rubtsov.
\newblock {\em {Contact geometry and nonlinear differential equations.}}
\newblock {Encyclopedia of Mathematics and Its Applications 101. Cambridge:
  Cambridge University Press. xxi, 496~p.}, 2007.

\bibitem{MVY}
Michal Marvan, Alexandre~M. Vinogradov, and Valery~A. Yumaguzhin.
\newblock Differential invariants of generic hyperbolic {M}onge-{A}mp\`ere
  equations.
\newblock {\em Cent. Eur. J. Math.}, 5(1):105--133, 2007.

\bibitem{Vin3}
A.~M. Vinogradov.
\newblock Scalar differential invariants, diffieties and characteristic
  classes.
\newblock In {\em Mechanics, analysis and geometry: 200 years after
  {L}agrange}, North-Holland Delta Ser., pages 379--414. North-Holland,
  Amsterdam, 1991.

\bibitem{Vin}
A.~M. Vinogradov.
\newblock {\em {Cohomological analysis of partial differential equations and
  secondary calculus. Transl. from the original Russian manuscript by Joseph
  Krasil'shchik.}}
\newblock {Translations of Mathematical Monographs 204. Providence, RI:
  American Mathematical Society (AMS). xv, 247~p.}, 2001.

\bibitem{Vin2}
A.~M. Vinogradov.
\newblock On the geometry of second-order parabolic equations with two
  independent variables.
\newblock {\em Dokl. Akad. Nauk}, 423(5):588--591, 2008.

\end{thebibliography}

\end{document}